\documentclass[11pt,reqno]{amsart}

\usepackage[dvipsnames]{xcolor}
\usepackage{tikz}
\usepackage{stmaryrd}
\usetikzlibrary{matrix,arrows,decorations.pathmorphing}
\usepackage{graphicx}
\usepackage{amssymb}
 \usepackage[stable]{footmisc}
\usepackage{amsmath,mathtools}
\usepackage{fullpage}
\usepackage{caption}
\usepackage{amsthm}
\usepackage{mathrsfs}
\usepackage{thmtools}
\usepackage{blkarray}
\usepackage{multirow}
\usepackage{picinpar} % for pictures in paragraphs
\usepackage{tikz-cd}
\usepackage{color}
\usepackage{verbatim}
\usepackage{hyperref}
\hypersetup{colorlinks=true, linkcolor=black, citecolor = OliveGreen, urlcolor = black}
\usepackage{amssymb}
\usepackage{amsmath}
\usepackage[shortlabels]{enumitem}
\usepackage{colonequals}
\usepackage{faktor}
\usepackage{color}
\usepackage{mathrsfs}
\usepackage[all]{xy}
\usepackage{comment}
\usepackage{mathpazo}
\usepackage{euler}
\usepackage[noabbrev,capitalise]{cleveref}

\usepackage[margin=1.in]{geometry}
\tikzset{
  closed/.style = {decoration = {markings, mark = at position 0.5 with { \node[transform shape, xscale = .8, yscale=.4] {/}; } }, postaction = {decorate} },
  open/.style = {decoration = {markings, mark = at position 0.5 with { \node[transform shape, scale = .7] {$\circ$}; } }, postaction = {decorate} }
}

\newcommand{\mZ}{\mathbb{Z}}

\newcommand{\mQ}{\mathbb{Q}}

\newcommand{\mC}{\mathbb{C}}
\newcommand{\mG}{\mathbb{G}}

\newcommand{\cA}{\mathcal{A}}

\newcommand{\cO}{\mathcal{O}}

\newcommand{\cB}{\mathcal{B}}

\usepackage[OT2,T1]{fontenc}
\DeclareSymbolFont{cyrletters}{OT2}{wncyr}{m}{n}
\DeclareMathSymbol{\Sha}{\mathalpha}{cyrletters}{"58}

\DeclareMathSymbol{\Sha}{\mathalpha}{cyrletters}{"58}

\newcommand{\lra}{\longrightarrow}

\newcommand{\brk}[1]{ \left\lbrace #1 \right\rbrace }
\newcommand{\pwr}[1]{ \left( #1 \right) }

\newcommand{\uu}{\underline{u}}

\RequirePackage{mathrsfs} %\let\mathcal\mathcal

\numberwithin{equation}{subsection}
\newtheorem{thmx}{Theorem}

\numberwithin{equation}{subsection}
\newtheorem{theorem}[subsection]{Theorem}

\newtheorem{lemma}[subsection]{Lemma}

\theoremstyle{definition}
\newtheorem{definition}[subsection]{Definition}

\newtheorem{assumption}[subsection]{Assumption}

\theoremstyle{remark}

\newtheorem{remark}[subsection]{Remark}

\numberwithin{equation}{section} \numberwithin{figure}{section}

\DeclareMathOperator{\per}{per}

\DeclareMathOperator{\Lie}{Lie}
\DeclareMathOperator{\Hom}{Hom}

\newcommand{\cdef}[1]{\textsf{\textit{#1}}}

\renewcommand{\leq}{\leqslant}
\renewcommand{\le}{\leqslant}
\renewcommand{\geq}{\geqslant}
\renewcommand{\ge}{\geqslant}

\DeclareMathOperator{\Cz}{Cz}
\DeclareMathOperator{\Fo}{F}

\newcommand\nc{\newcommand}
\nc\mf\mathfrak
\nc\mc\mathcal
\nc\mb\mathbb
\nc\msf\mathsf

\makeatletter
\@namedef{subjclassname@1991}{\emph{2020} Mathematics Subject Classification}
\makeatother

\newcommand{\Z}{{\mathbb Z}}
\newcommand{\Q}{{\mathbb Q}}
\newcommand{\C}{{\mathbb C}}
\newcommand{\R}{{\mathbb R}}

\newcommand{\N}{{\mathbb N}}

\newcommand{\Kbar}{{\overline{K}}}

\newcommand{\cF}{{\mathcal F}}

\newcommand{\barM}{\overline{M}}

\makeatletter
\let\@wraptoccontribs\wraptoccontribs
\makeatother

\begin{document}

\title{On $p$-adic uniformization of abelian varieties with good reduction}

\author{Adrian Iovita}
\address{Adrian Iovita \\
	Concordia University \\
	Department of Mathematics and Statistics \\ 
	Montr\'eal, Qu\'ebec, Canada and 
Dipartimento di Matematica \\
Universita degli Studi di Padova \\
Padova, Italy
}
\email{adrian.iovita@concordia.ca}

\author{Jackson S. Morrow}
\address{Jackson S. Morrow \\
	Department of Mathematics \\
	University of California, Berkeley \\
	749 Evans Hall, Berkeley, CA 94720}
\email{jacksonmorrow@berkeley.edu}

\author{Alexandru Zaharescu}
\address{Alexandru Zaharescu\\
Department of Mathematics \\
University of Illinois Urbana-Champaign \\
1409 West Green Street, Urbana, IL 61801, USA and
''Simion Stoilow'' Institute of Mathematics of the Romanian Academy\\
P.O. Box 1-764, RO-014700 Bucharest, Romania.
}
\email{zaharesc@illinois.edu}

	% Formal groups, p-divisible groups

\date{\today}

\contrib[with an appendix by]{Yeuk Hay Joshua Lam and Alexander Petrov}

\begin{abstract}
Let $p$ be a rational prime, let $F$ denote a finite, unramified extension of $\Q_p$, $K$ the maximal unramified extension of $\Q_p$, $\Kbar$ some fixed algebraic closure of $K$, and $\C_p$ the completion of $\Kbar$. Let $G_F$ the absolute Galois group of $F$. 
Let $A$ be an abelian variety defined over $F$, with good reduction. 
Classically, the Fontaine integral was seen as a Hodge--Tate comparison morphism, i.e.~as a map $\varphi_{A} \otimes 1_{\C_p}\colon T_p(A)\otimes_{\Z_p}\C_p\to \Lie(A)(F)\otimes_F\C_p(1)$, and as such it is surjective and has a large kernel.

The present article starts with the observation that if we do not tensor $T_p(A)$ with $\C_p$, then the Fontaine integral is often injective. 
In particular, it is proved that if $T_p(A)^{G_K} = 0$, then $\varphi_A$ is injective.  
As an application,  we extend the Fontaine integral to a perfectoid like universal cover of $A$ and show that if $T_p(A)^{G_K} = 0$, then $A(\overline{K})$ has a type of $p$-adic uniformization, which resembles the classical complex uniformization.  
\end{abstract}
\maketitle
\vspace*{-.8cm}
\section{\bf Introduction}\label{sec:intro}
As a consequence of a series of articles starting with Tate's original \cite{tate:pdiv} and ending with the article of Scholze--Weinstein \cite{scholze_weinstein}, we now have a complete classification of the $p$-divisible groups over $\cO_{\C_p}$ in terms of their Hodge--Tate structures.
In this paper, we investigate the question of whether a similar classification exists for abelian varieties over a $p$-adic field from a different angle, closer to the classical uniformization of abelian varieties over $\mathbb{C}$. 

This work has two main inspiration sources:~\cite{fontaine_presque} and \cite{bloch_kato}. 
To start with, let us recall that in the Introduction, in fact in the Prologue of \cite{fontaine_presque}, J.-M. Fontaine presents the following picture. 
Let $K$ denote the completion of a number field and choose an algebraic closure $\Kbar$ of $K$. We denote the absolute Galois group of $K$ by $G_K:={\rm Gal}(\Kbar/K)$ and by $C$ the completion of $\Kbar$. Let $A$ denote an abelian variety over $K$ of dimension $g\ge 1$, and let ${\rm Lie}(A)$ denote the Lie algebra of $A$, seen as a vector group over $K$.

\subsubsection*{Let us first suppose $K$ is archimedean}
 In this case we have $K=\R$ or $K=\C$, $\Kbar=\C=C$, $G_K=\{1, \tau\}$, with $\tau$ the complex conjugation, or $G_K=\{1\}$. 
The exponential map is everywhere defined and defines the exact sequence:
\[
0\to H_1(A(\C), \Z) \to \Lie(A)(\C)\stackrel{\exp_A}{\to} A(\C)\to 0.
\]
From this, we obtain the complex uniformization of $A$, i.e. an isomorphism of complex Lie-groups: $$A(\C)\cong \Lie(A)(\C)/\Lambda\cong \C^g/\Lambda,$$ where
$\Lambda$ is the image of the lattice $H_1(A(\C), \Z)$.

\subsubsection*{Now suppose that $K$ is a finite extension of $\Q_p$ for $p>2$ a rational prime}
In this case, we will denote $\C_p:=C$.
Then the logarithm is everywhere defined and we have a natural commutative diagram with exact rows:
\[
\begin{tikzcd}[row sep = 1em]
0 \arrow{r} & A_{\rm tor}(\Kbar)\arrow{r}\arrow[equal]{d} &A(\Kbar) \arrow{r}{\log_A} \arrow[right hook->]{d}& {\rm Lie}(A)(\Kbar) \arrow{r} \arrow[right hook->]{d}&0\\
0 \arrow{r} & A_{\rm tor}(\C_p)\arrow{r} &A(\C_p) \arrow{r}{\log_A} & {\rm Lie}(A)(\C_p) \arrow{r} &0.
\end{tikzcd}
\]
The group of points $A(\Kbar)$ has a natural topology (see Subsection \ref{sec:topologies}) whose completion is $A(\C_p)$ and which induces the discrete topology on $A_{\rm tor}(\Kbar)$ and the natural, $p$-adic topology on ${\rm Lie}(A)(\Kbar)$.

Let $A_{\rm tor}(\Kbar)$ denote the torsion subgroup of $A(\Kbar)$. We have the following decomposition
\[
A_{\rm tor}(\Kbar)=A_{p{\text{-tors}}}(\Kbar)\oplus A_{p'{\text{-tors}}}(\Kbar)
\] 
into the $p$-power torsion $A_{p{\text{-tors}}}(\Kbar)$ and the prime to $p$-torsion $A_{p'{\text{-tors}}}(\Kbar)$. 
Fontaine has constructed a section $s\colon A(\C_p)\to A_{p'{\text{-tors}}}(\Kbar)$ of the natural inclusion $A_{p'{\text{-tors}}}(\Kbar)\subset A(\C_p)$ and we denote the kernel of $s$ by $A^{(p)}(\C_p)$ (see Section \ref{sec:uniformization} or \cite{fontaine_presque} for more details on the construction of $s$).  
As such, we have the decompositions 
\[
A(\C_p)=A^{(p)}(\C_p)\oplus A_{p'{\text{-tors}}}(\Kbar)
\]
and
\[
A(\Kbar)=A^{(p)}(\Kbar)\oplus A_{p'{\text{-tors}}}(\Kbar)
\] 
where we let $A^{(p)}(\Kbar):=A^{(p)}(\C_p)\cap A(\Kbar)$.

Fontaine states, as a remarkable fact, that we can recover $A^{(p)}(\Kbar)$, and so also $A^{(p)}(\C_p)$, as topological abelian groups with $G_K$-action, from the knowledge of $A_{p{\text{-tors}}}(\Kbar)$. Moreover, this implies that we can recover $A(\Kbar)$, respectively $A(\C_p)$ from the knowledge of $A_{\rm tor}(\Kbar)$ (cf.~\cite[Proposition 1.1]{fontaine_presque}).

\subsection{Main contributions}
The present article has as objective to show that, under certain circumstances, by changing the topology of $A^{(p)}(\Kbar)$, this group can be determined in a different way from $A_{p{\text{-tors}}}(\Kbar)$. 
First, we recall that in \cite{fontaine:differentials}, Fontaine constructed an integration map
\[
\varphi_A\colon T_p(A) \to \Lie(A)(K)\otimes_K\C_p(1),
\]
and when this map was tensored with $\C_p$, it realizes the Hodge--Tate comparison morphism. In particular, the map 
\[
\varphi_A \otimes 1_{\C_p}\colon T_p(A)\otimes_{\Z_p}\C_p\to \Lie(A)(K)\otimes_K\C_p(1),
\]
is surjective and it has a large kernel.
A starting point for this article is the observation that if we do not tensor $T_p(A)$ with $\C_p$, then the Fontaine integral is often injective.

To state the results precisely, we need to establish some notation. Let now $K$ denote the maximal unramified extension of $\Q_p$, let $A$ denote an abelian variety defined over some subfield $F\subset K$ such that $[F:\Q_p]<\infty$, with good reduction over 
$F$. 
Let $\cA$ denote the N\'eron model of $A$ over ${\rm Spec}(\cO_F)$.

We present a proof, whose sketch was supplied by Pierre Colmez, that if $T_p(A)^{G_K}=0$, then Fontaine's integration map $\varphi_A\colon T_p(A)\to {\rm Lie}(A)(F)\otimes_F\C_p(1)$ is injective. There is another proof of this fact in the Appendix by Yeuk Hay Joshua Lam and Alexander Petrov (independently).

In fact, we have more; to describe this, we need to briefly recall some definitions. 
We define the universal covering space of $A(\overline{K}) = \cA(\cO_{\Kbar})$ to be
\begin{align*}
\cB_\cA & :=\varprojlim\left(\cA(\cO_{\Kbar})\stackrel{[p]}{\longleftarrow}\cA(\cO_{\Kbar})\stackrel{[p]}{\longleftarrow}\cdots \stackrel{[p]}{\longleftarrow}\cA(\cO_{\Kbar})\cdots\right),
\end{align*}
and we call an element $\uu = (u_n)_{n\geq 0}$ of $\cB_{\cA}$ a path. 
It is clear that $\cB_\cA$ is a $G_K$-module which sits in the following exact sequence:
\[
0\to T_p(A)\cong T_p(\cA)\to \cB_\cA\stackrel{\alpha}{\to}\cA(\cO_{\Kbar})\to 0,
\]
with $\alpha\left((u_n)_{n\ge 0}  \right):=u_0$ for all $\underline{u}=(u_n)_{n\ge 0}\in \cB_\cA$. 
In \autoref{defn:Fontaineintegral}, we extend the classical Fontaine integral to a non-zero, $G_F$-equivariant map 
\[
\varphi_{\cA}\colon \cB_{\cA}\to \Lie(\cA)(\cO_F)\otimes_F\C_p(1),
\]
and with this definition, it is clear that if a path $\uu = (u_n)_{n\geq 0}$ is periodic (i.e., there exists some $k\geq 1$ such that $u_0 = u_k$), then $\varphi_{\cA}(\uu)(\omega)  = 0$. 
Our first result proves that the kernel of this extended Fontaine integral is precisely the periodic paths. 

\begin{thmx}\label{conj:periodic1}
Let $A$ be an abelian variety over $F$ with good reduction, and let $\cA$ denote its N\'eron model. 
Suppose that $T_p(A)^{G_K}=0$. Then, the kernel of the Fontaine integral $\varphi_\cA$ extended to $\cB_{\cA}$ is precisely the subgroup of periodic paths of $\cB_{\cA}$ (\autoref{defn:periodicpaths}). 
\end{thmx}
    
As an application of \autoref{conj:periodic1}, we show that if $A$ satisfies $T_p(A)^{G_K} = 0$, then $A^{(p)}(\Kbar)$ can be determined in a different way from $A_{p{\text{-tors}}}(\Kbar)$.

\begin{thmx}\label{thm:main2}
Let $K:=\Q^{\rm ur}_p$ denote the maximal unramified extension of $\Q_p$. 
Let $A$ denote an abelian variety defined over some subfield $F\subset K$ such that $[F:\Q_p]<\infty$, with good reduction over 
$F$. 
Let $\cA$ denote the N\'eron model of $A$ over ${\rm Spec}(\cO_F)$.

Suppose that $A$ satisfies $T_p(A)^{G_K} = 0$. Then, there exists a canonical, injective, continuous, $G_F$-equivariant map
\[
\iota_A\colon A^{(p)}(\Kbar)\hookrightarrow \left({\rm Lie}(A)(F)\otimes_F\C_p(1)\right)/\varphi_{\cA}(T_p(A)),
\]
where the topology on $A^{(p)}(\Kbar)$ which makes this embedding continuous is the $w$-topology (\autoref{defn:Apwithw}).

Moreover, an element $x\in \left(\Lie(\cA)(\cO_F)\otimes_{\cO_F}\C_p(1)\right)/\varphi_{\cA}(T_p(A))$ lies in the image of $\iota_A$ if and only if $x$ is crystalline (\autoref{defn:cyrstalline}). In particular, one can recover $A^{(p)}(\Kbar)$ from the triple
\[
\left(T_p(A), \Lie(A)(F)\otimes_{F}\C_p(1), \varphi_\cA\colon T_p(A)\hookrightarrow {\rm Lie}(A)(F)\otimes_{F}\C_p(1)\right). 
\] 
\end{thmx}

\begin{remark}
The fact that the field $F$ of definition of the abelian variety $A$
in the above theorem, is unramified is only used in Section \ref{sec:classes} and in the Appendix. It seems clear that with more work all the results of this paper, suitably adjusted, should hold when $F$ is a finite extension of $\Q_p$.
\end{remark}

\begin{remark}
We also, strangely enough, are able to prove analogues of the above results, on the kernel of the Fontaine integral and the uniformization, for the multiplicative group (cf.~Section \ref{sec:uniformizationmultiplicativegroup}).
\end{remark}

\subsection{Related results}
Our \autoref{thm:main2} states that if $A$ satisfies $T_p(A)^{G_K} = 0$, then $A(\overline{K})$ has a type of $p$-adic uniformization, which resembles the classical complex uniformization. 
The history of $p$-adic uniformization of abelian varieties is rich and beautiful, and we briefly exposit it below. 

The first work in this area was due to Tate \cite{Tate:Curve} who showed that an elliptic curve with multiplicative reduction is uniformized, as a rigid analytic space, by the rigid analytification of the multiplicative group. 
Later, Raynaud \cite{Raynaud:Uniformization} and Bosch--L\"utkebohmert \cite{BLI, BLII} extented Tate's result to arbitrary abelian varieties and isolated a class of abelian varieties, namely those with toric reduction, whose topological uniformization resembles complex uniformization. Subsequent developments in $p$-adic geometry by Berkovich \cite{BerkovichSpectral, BerkovichEtaleCohomology} gave rise to $p$-adic analytic spaces which have topological properties similar to those of complex manifolds. 
Using this theory, Berkovich \cite{BerkovichUniversalCover} showed that a smooth, connected Berkovich analytic space has a topological universal cover, and this result was later generalized by Hrushovshi--Loeser \cite{hrushovski2016non} to quasi-projective Berkovich analytic spaces. 
While Berkovich spaces have nice topological properties, adic spaces (in the sense of Huber \cite{huber2}) do not (e.g., they do not possess paths). 
That being said, one can use the theory of perfectoid spaces \cite{ScholzePerfectoid} and diamonds \cite{scholze2020berkeley, ScholzeDiamonds} to construct certain pro-\'etale uniformizations of the adic space associated to an abelian variety (c.f.~\cite{ blakestadetal:perfectoid, heuer:proetale}). 
In addition to these results, we mention work of Fargues \cite{Fargues:Pdivisible} and Scholze--Weinstein \cite{scholze_weinstein} which provide $p$-adic uniformization results for $p$-divisible groups.  The shape of \cite[Theorem B]{scholze_weinstein} is similar to \autoref{thm:main2}, but \autoref{thm:main2} is closer to the classical complex uniformization. 

To conclude,  we believe that \autoref{thm:main2} is the first result which gives a $p$-adic uniformization of abelian varieties with good reduction which is similar to the classical complex uniformization.

\subsection{Organization of paper}
In Section \ref{sec:Felliptic}, we recall the construction of Fontaine integration and discuss the role this integration theory plays in the context of Hodge--Tate and de Rham comparison isomorphisms for abelian varieties. 
In Section \ref{sec:uniformizationmultiplicativegroup}, we discuss the analogues of \autoref{conj:periodic1} and \autoref{thm:main2}  to the multiplicative group as a warm up for the case of abelian varieties. 
In Section \ref{sec:conjecture}, we state \autoref{conj:periodic1} as well as how it relates to previous literature. 
In Section \ref{sec:evidence}, we show that  \autoref{conj:periodic1} can be reduced to the injectivity of the Fontaine integral restricted to the Tate module of the formal group and prove this, following a sketch provided by Pierre Colmez.   
In Section \ref{sec:uniformization}, we prove \autoref{thm:main2}, which is our $p$-adic uniformization result for $A^{(p)}(\Kbar)$.

\subsection{Conventions}\label{subsec:conventions}
We establish the following notations and conventions throughout the paper.

\subsubsection*{Fields}
Fix a rational prime $p>2$.  
Let $K:=\Q^{\rm ur}_p$ denote the maximal unramified extension of $\Q_p$, let $\overline{K}$ be a fixed algebraic closure of $K$, and let $\C_p$ denote the completion of $\overline{K}$ with respect to the unique extension $v$ of the $p$-adic valuation on $\Q_p$ (normalized such that $v(p) = 1$). 
For a tower of field extensions $\Q_p\subset F\subset K$, we denote by $G_K$ and respectively $G_F$ the absolute Galois groups of $K$ and $F$ respectively. %Note that $G_K$ is the inertia subgroup of $G_F$, the reason we %denote it sometimes by $I_F$.  
We denote $\cO:=\cO_{\Kbar}$.
We remark that working over $K$, as we do in this article, is not essential. One could start by fixing a finite extension $L$ of $\Q_p$, and end-up working on $L^{\rm ur}$.

\subsubsection*{Abelian varieties}
We will consider an abelian variety $A$ defined over some  subfield $F\subset K$ such that $[F:\Q_p]<\infty$, with good reduction over 
$F$.
%We will simply say that $A$ is an abelian variety over $K$ with good reduction. 
Let $\cA$ denote the N\'eron model of $A$ over ${\rm Spec}(\cO_F)$ and also denote by $\widehat{\cA}$ the formal completion of $\cA$ along the identity of its special fiber, i.e. the formal group of $A$.
We note that the formation of N\'eron models commutes with unramified base change. 
We will denote the Tate module of $A$ (resp.~the N\'eron model $\cA$ of $A$) by $T_p(A)$ (resp.~$T_p(\cA)$).
We note that $T_p(A)\cong T_p(\cA)$ as $G_F$-modules.

We recall that $\widehat{\cA}$ is a formal group of dimension $\dim (A)$ and of height $h$ which satisfies $\dim(A) \leq h \leq 2\dim(A)$.

\subsection*{Acknowledgements}
This article owes much to many people. 
We thank Robert Benedetto, Olivier Brinon, Henri Darmon, Eyal Goren, Ralph Greenberg, and Sean Howe for helpful conversations and email exchanges on topics related to this research. 
We are grateful to Pierre Colmez for sending us a sketch of the proof of \autoref{thm:Fontaineformalinjective} and to Yeuk Hay Joshua Lam and Alexander Petrov for providing us with the proof presented in the Appendix. 
We also thank Pierre Colmez and Jan Nekovar for pointing out some errors in earlier drafts of this paper.
We would like to thank the referees for their very helpful and detailed comments.
This research began at the thematic semester on  “Number Theory --- Cohomology in Arithmetic” at the Centre de Recherches Math\'ematiques. 
Finally, the first author was partially supported by an NSERC Discovery grant.

\section{\bf Fontaine integration for abelian varieties with good reduction}
\label{sec:Felliptic}
In this section, we recall the construction of the Fontaine integration and the extension of this integration theory to a certain universal cover of an abelian variety. 
We also describe several topological aspects of the integration theory and discuss how these integration theories realize the Hodge--Tate and de Rham comparison isomorphisms for abelian varieties.

\subsection{The differentials of the algebraic integers}
First, we recall for the reader's convenience the notation established above. 
Let $K:=\Q^{\rm ur}_p$ denote the maximal unramified extension of $\Q_p$, let $\overline{K}$ be a fixed algebraic closure of $K$, and let $\C_p$ denote the completion of $\overline{K}$. 
Let $G_K$ denote the absolute Galois group of $K$. 
We denote $\cO:=\cO_{\Kbar}$. Fix a finite extension $F$ of $\Q_p$ in $K$.
For a $G_K$-representation $V$, the $n$-th Tate twist of $V$ is denoted by $V(n)$, which is just the tensor product of $V$ with the $n$-fold product of the $p$-adic cylcotomic character $\mQ_p(1)$. 

In \cite{fontaine:differentials}, Fontaine studied a fundamental object related to these choices, namely the $\cO$-module $\Omega:=\Omega^1_{\cO/\cO_K}\cong \Omega^1_{\cO/\cO_F}$ of K\"ahler differentials of $\cO$ over $\cO_K$, or over $\cO_F$. 
The $\cO$-module $\Omega$ is a torsion and $p$-divisible $\cO$-module, with a semi-linear action of $G_F$. 
Let $d\colon\cO\to \Omega$ denote the canonical derivation, which is surjective.

Important examples of algebraic differentials arise as follows:~Let $(\varepsilon_n)$ denote a compatible sequence of primitive $p$th roots of unity in $\overline{K}$. 
Then 
\[
\frac{d\varepsilon_n}{\varepsilon_n} = d (\log \varepsilon_n) \in \Omega \quad \mbox{ and } \quad  p\left(\frac{d\varepsilon_{n+1}}{\varepsilon_{n+I}}\right)=\frac{d\varepsilon_n}{\varepsilon_n}. 
\]

Next, we recall a theorem of Fontaine.

\begin{theorem}[\protect{\cite[Th\'eor\`eme 1']{fontaine:differentials}}]\label{thm:Fontainedifferentials}
Let $(\varepsilon_n)$ denote a compatible sequence of primitive $p$th roots of unity in $\overline{K}$. 
For $\alpha \in K$, write $\alpha = a/p^r$ for some $a\in \mathcal{O}$. 
The morphism $\xi\colon \overline{K}(1) \to \Omega$ defined by 
\[
\xi(\alpha \otimes (\varepsilon_n)_n) = a \frac{d\varepsilon_r}{\varepsilon_r}
\]
is surjective and $G_K$-equivariant with kernel 
\[
\ker(\xi) = \underline{a}_{K} \coloneqq \brk{x\in \overline{K} : v(x) \geq - \frac{1}{p-1}}.
\]
Moreover, $\Omega\cong{\overline{K}(1)}/{\underline{a}_K(1)} \cong (\overline{K}/\underline{a}_K)(1)$ and $V_p(\Omega) = \Hom_{\mZ_p}(\mQ_p,\Omega) \cong \mC_p(1)$.
\end{theorem}

\autoref{thm:Fontainedifferentials} implies the following:
\[
T_p(\Omega)\otimes_{\Z_p}\Q_p:=\left(\varprojlim_n \Omega[p^n]\right) \otimes_{\Z_p}\Q_p\cong \left(\varprojlim\left(\Omega\stackrel{p}{\leftarrow}\Omega\stackrel{p}{\leftarrow}\cdots\stackrel{p}{\leftarrow}\Omega\cdots\right)\right)\otimes_{\Z_p}\Q_p\cong\C_p(1)
\]
as $G_F$-modules.

We denote by $\cO^{(1)}:=\ker(d)$, the kernel of $d$, which is an $\cO_K$-sub-algebra of $\cO$. 
Indeed, if $a,b\in \cO^{(1)}$, then $d(ab)=ad(b)+bd(a)=0$, and so $ab\in \cO^{(1)}$. 
We note that the derivation map $d\colon \cO \to \Omega$ is not continuous with respect to the $p$-adic topology on $\cO$ and the discrete topology on $\Omega$. 
Nevertheless, there is another topology on $\cO$ for which $d$ is continuous, namely the topology defined by the map $w\colon\cO\lra \Z\cup \{\infty\}$ defined by $w(a):=\sup\{n\in \mZ \mid a\in p^n\cO^{(1)}\}$. 
%, where $\cO^{(1)}=\ker(d)$.  
We summarize this result in the following lemma.

\begin{lemma}\label{lemma:continuousd}
Consider $\cO$ with the topology defined by $w$ and endow $\Omega$ with the discrete topology. 
Then, the derivation map $d\colon \cO \to \Omega$ is continuous, and moreover, this topology is the weakest topology for which $d$ is continuous. 
\end{lemma}

\begin{remark}
The topology defined by $w$ on $\cO$ is finer than the topology defined by $p$-adic valuation $v_p$; in particular, anything discrete for $v_p$ is also discrete for $w$ (c.f.~\cite[Proposition 1.4.3]{Fontaine:PadicRep}).  
Also, we note that by \cite[Theorem 1]{colmez} the completion of $\overline{K}$ for $w$ is $B_{\text{dR}}^+/\text{Fil}^2$. 
\end{remark}

\subsection{A universal cover of an abelian variety}
Classically, the Fontaine integral is defined as a map
\[
\varphi_A\colon T_p(A) \to \Lie(A)(F)\otimes_F\C_p(1).
\]
For our purposes, we will want to integrate along a larger collection of paths, which leads to the following definition. 
We note that this construction appears in \cite[Part II, Section 3]{coleman_iovita}. 

\begin{definition}
\label{defn:universalcoveringspace}
We define the \cdef{universal covering space of $A(\overline{K})$} to be
\begin{align*}
\cB_\cA & :=\varprojlim\left(\cA(\cO)\stackrel{[p]}{\longleftarrow}\cA(\cO)\stackrel{[p]}{\longleftarrow}\cdots \stackrel{[p]}{\longleftarrow}\cA(\cO)\cdots\right) = \varprojlim\left(A(\overline{K})\stackrel{[p]}{\longleftarrow}A(\overline{K})\stackrel{[p]}{\longleftarrow}\cdots \stackrel{[p]}{\longleftarrow}A(\overline{K})\cdots\right)
\end{align*}
where the equality comes from the properness of $\cA$. We call an element $\uu = (u_n)_{n\geq 0}$ of $\cB_{\cA}$ \cdef{a path}. 
\end{definition}

We immediately have that $\cB_\cA$ is a $G_F$-module which sits in the following exact sequence:
\begin{equation}\label{eqn:SESpaths}
0\to T_p(A)=T_p(\cA)\to \cB_\cA\stackrel{\alpha}{\to}\cA(\cO)\to 0,
\end{equation}
with $\alpha\left( \uu \right):=u_0$ for all $\underline{u}=(u_n)_{n\ge 0}\in \cB_A$. 

\begin{remark}
The universal covering space from \autoref{defn:universalcoveringspace} is similar to the construction of the perfectoid universal cover of an abelian variety (see e.g., \cite[Lemma 4.11]{pilloni2016cohomologie} and \cite{blakestadetal:perfectoid}). 
We refrain from using this terminology as we work over the field $\overline{K}$, which is not perfectoid, and hence $\cB_{\cA}$ cannot come from any perfectoid space. 
We also mention that \cite[Appendice B]{Colmez:Integration} discusses universal covering spaces of abelian varieties over $p$-adic fields. 
\end{remark}

\subsection{The definition of Fontaine integration}\label{sec:Fontaineintegral}
We are now ready to define Fontaine integration. 
Let $ H^0(\cA, \Omega^1_{\cA/\cO_F})$ and respectively $\Lie(\cA)(\cO_F)$ denote the $\cO_F$-modules of invariant differentials on $\cA$ and respectively its Lie algebra. 
Note that $\omega\in  H^0(\cA, \Omega^1_{\cA/\cO_F})$ being invariant implies that $(x\oplus_{
\cA} y)^*(\omega) = x^*(\omega) + y^*(\omega)$ and $[p]^*(\omega) = p\omega$ where $\oplus_{
\cA}$ is the group law in $A(\Kbar)$. 

\begin{definition}\label{defn:Fontaineintegral}
Let $\underline{u}=(u_n)_{n\in \N}\in \cB_\cA$ and $\omega\in H^0(\cA, \Omega^1_{\cA/\cO_F})$.
Each $u_n\in \cA(\cO)$ corresponds to a morphism $u_n\colon {\rm Spec}(\cO)\to \cA$, and hence we can pullback $\omega$ along this map giving us a K\"ahler differential $u_n^*(\omega)\in \Omega$. 
The sequence $\left(u_n^*(\omega)\right)_{n\geq 0}$ is a sequence of differentials in $\Omega$ satisfying
$pu_{n+1}^\ast(\omega)=u_n^\ast(\omega)$, and hence defines an element in $V_p(\Omega)\cong \C_p(1)$.

The \cdef{Fontaine integration map} 
\[
\varphi_\cA\colon \cB_\cA\to \Lie(\cA)(\cO_F)\otimes_{\cO_F}\C_p(1)
\] 
is a non-zero, $G_F$-equivariant map defined by
\[
\varphi_\cA(\underline{u})(\omega):=\left(u_n^\ast(\omega)\right)_{n\ge 0}\in V_p(\Omega)\cong \C_p(1).
\]
\end{definition}

\subsection{Topological aspects of the extended Fontaine integral}\label{sec:topologies}
Recall the notation from Subsection \ref{subsec:conventions}. We have two $G_F$-modules, namely $A(\Kbar)=\cA(\cO)$ and $\cB_\cA$, which was defined in \autoref{defn:universalcoveringspace}.
A priori, $A(\Kbar)=\cA(\cO)$ and $\cB_\cA$ are just abelian groups, but we can endow them with natural topologies to enhance them to topological abelian groups as follows.

Fontaine defines the following natural topology on $A(\Kbar)$,  namely the topology for which a basis of neighborhoods of $x \in \cA(\cO)$ is of the form $x\oplus_{\cA} \widehat{\cA}(p^n\cO)$ for $n\geq 1$ where $\oplus_{
\cA}$ is the group law in $A(\Kbar)$. 
We denote $\cA(\cO)$ with this topology by $\cA^{\text{Fo}}(\cO)$, which makes $A(\Kbar) = \cA(\cO)$ into a topological abelian group and induces the discrete topology on $A_{\rm tor}(\Kbar)$, the subgroup of torsion points of $A(\Kbar)$ and the $p$-adic topology on the points of $\widehat{\cA}$. 
Let $\psi_{\cA}\colon \cA(\cO) \to \Lie(\cA)(\cO_F) \otimes_{\cO_F} \Omega$ denote the Fontaine integral i.e., defined by sending $a\in \cA(\cO)$ to $\psi_{\cA}(a)(\omega) \coloneqq a^*(\omega) \in \Omega$ where $\omega \in H^0(\cA,\Omega_{\cA/\cO_F}^1)$. 
As the derivative map $d$ is not continuous with respect to the natural $p$-adic topology on $\cO$, we see that the map $\psi_{\cA}$ is not continuous for $\cA^{\text{Fo}}(\cO)$ and the discrete topology on $\Lie(\cA)(\cO_F) \otimes_{\cO_F} \Omega$.  

In order to get continuity of the Fontaine integral $\psi_{\cA}$ and the extended Fontaine integral $\varphi_{\cA}$, we will need to define a new topology on which resembles the $w$-topology from \autoref{lemma:continuousd}. 

\begin{definition}
The \cdef{$w$-topology} on $\cA(\cO) = A(\Kbar)$ is defined to be the topology for which a basis of neighborhoods of $x \in \cA(\cO)$ is of the form $x\oplus_{\cA} \widehat{\cA}(p^n\cO^{(1)})$ for $n\geq 1$ where $\oplus_{
\cA}$ is the group law in $A(\Kbar)$. 
We denoted $\cA(\cO)$ with the $w$-topology by $\cA^w(\cO)$. 
\end{definition}

\begin{lemma}\label{lemma:winducesdiscretetorsion}
\begin{enumerate}
\item []
\item The map ${\rm Id}\colon \cA^w(\cO)\to \cA(\cO)$ is continuous, where on the target of the map we have the usual $p$-adic topology, i.e. the linear topology in which a basis of neighbourhoods of $0$ is $\bigl( \widehat{\cA}(p^n\cO)\bigr)_{n \geq 1}$. 
\item The $w$-topology on $\cA(\cO)$ induces the discrete topology on $\cA(\cO)_{\rm tor}$.
\end{enumerate}
\end{lemma}

\begin{proof}
For (1), it is enough to see that for every $n\ge 0$ the subset $\widehat{\cA}(p^n\cO)\subset \cA(\cO)$ is open in the $w$-topology.
Let $\cF\subset \widehat{\cA}(p^n\cO)$ be a system of representatives of the quotient $\widehat{\cA}(p^n\cO)/\widehat{\cA}(p^n\cO^{(1)})$. Then $\widehat{\cA}(p^n\cO)=\cup_{x\in \cF}\bigl(x\oplus \widehat{\cA}(p^n\cO^{(1)})\bigr)$ and therefore it is open in $\cA^w(\cO)$. 
For (2), we note that as the $p$-adic topology on $\cA(\cO)$ induces the discrete topology on $\cA(\cO)_{\rm tor}$, part (1) implies that every torsion point of $\cA(\cO)$ is open in the $w$-topology of $\cA(\cO)$, therefore $\cA(\cO)_{\rm tor}$ is discrete in $\cA^w(\cO)$.  
\end{proof}

To conclude this discussion, we will show that the action of $G_F$ on $\cA^w(\cO)$ and the Fontaine integral $\psi_{\cA}$ defined on $\cA^w(\cO)$ are continuous. 

\begin{lemma}\label{lemma:GKcontinuousw}
The action of $G_F$ on $\cA^w(\cO)$ is continuous.
\end{lemma}

\begin{proof}
We observe that for every $\sigma\in G_F$ we have $\sigma(\widehat{\cA}(p^n\cO^{(1)})=\widehat{\cA}(p^n\cO^{(1)})$, which implies that $\sigma\colon\cA(\cO)\to \cA(\cO)$ is continuous for the $w$-topology on $\cA(\cO)$. 
 
 We denote by $s\colon G_F\times {\cA}(\cO)\to \cA(\cO)$ the action of $G_F$ on $\cA(\cO)$ i.e. , $s(\sigma, y)=\sigma(y)$, for $\sigma\in G_F$ and $ y\in \cA(\cO)$.
  Let now $x\in \cA(\cO)$ and $n\ge 0$.
  Suppose $(\sigma, y)\in s^{-1}(x\oplus \widehat{\cA}(p^n\cO^{(1)}))$ and let $U_{(\sigma, y)}:=\{\tau\in G_F\, |\, \tau(y)=\sigma(y)  \}$, then
  $U_{(\sigma, y)}$ is open in $G_F$ and $(\sigma, y)\in U_{(\sigma, \tau)}\times (y\oplus \widehat{\cA}(p^n\cO^{(1)}))\subset s^{-1}(x\oplus \widehat{\cA}(p^n\cO^{(1)}))$.  
  This implies $s$ is continuous. 
  \end{proof}

\begin{lemma}\label{lemma:Fontainecontinuousw}
The Fontaine integral $\psi_{\cA}\colon \cA^w(\cO)\times H^0(\cA, \Omega^1_{\cA/\cO_F})\to \Omega$ defined by
$(x, \omega) \to x^*(\omega)\in \Omega$ is continuous.
\end{lemma}

\begin{proof}
We recall that on $\cA^w(\cO)$ we have the $w$-topology, on $H^0\bigl(\cA, \Omega^1_{\cA/\cO_F}  \bigr)$ we have the $p$-adic topology, and on $\Omega$ we have the discrete topology.

Let $x\in \cA(\cO)$ and $\omega\in H^0\bigl(\cA, \Omega^1_{\cA/\cO_F}  \bigr)$ and denote by $\eta\coloneqq x^*(\omega)\in \Omega$.
Let $U\coloneqq \psi_{\cA}^{-1}(\eta)\subset \cA^w(\cO)\times H^0\bigl(\cA, \Omega^1_{\cA/\cO_F}  \bigr)$. 
We claim 
$U$ is a neighborhood of $(x, \omega)$. 
For this let $V_x:=x\oplus_{\cA}\widehat{\cA}\bigl(p^m\cO^{(1)}\bigr)$ for some $m\ge 1$
be a neighbourhood of $x$. Let $\omega_1,\dots, \omega_g$ denote an $\cO_F$-basis of $H^0\bigl(\cA, \Omega^1_{\cA/\cO_F}  \bigr)$ and let $n\ge 1$ be an integer such that $p^nx^*(\omega_i)=0$ for $i=1,2,\dots,g$. Then for any $\beta\in H^0\bigl(\cA, \Omega^1_{\cA/\cO_F}  \bigr)$, we have $p^nx^*(\beta)=0$.  Let $V_\omega\coloneqq \omega+p^nH^0\bigl(\cA, \Omega^1_{\cA/\cO_F}  \bigr)$ denote the neighborhood of $\omega$.
We claim that $V_x\times V_\omega\subset U$. 

Let $(y, \gamma)\in V_x\times V_\omega$, i.e. ~$y=x\oplus_{\cA}z$, with $z\in \widehat{\cA}\bigl(p^m\cO^{(1)}\bigr)$ and $\gamma=\omega+p^n\beta$, with $\beta\in H^0\bigl(\cA, \Omega^1_{\cA/\cO_F}  \bigr)$.
We have
\[
(x\oplus_{\cA}z)^*(\omega+p^n\beta)=x^*(\omega+p^n\beta)+z^*(\omega+p^n\beta)=x^*(\omega)+p^nx^*(\beta)=x^*(\omega)=\eta,
\]
which gives our desired result. 
Note that the first equality holds because the differential $\omega+p^n\beta$ is invariant. 
For the second equality, we have, denoting $\delta:=\omega+p^n\beta\in H^0\bigl(\cA, \Omega^1_{\cA/\cO_F}  \bigr)$,  that
\[
z^*(\delta)=z^*(\delta|_{\widehat{\cA}})=(z_1,z_2,...,z_g)^*\left(\sum_{i=1}^gF_i(X_1,...,X_g)dX_i  \right)=\sum_{i=1}^gF_i(z_1,\dots,z_g)dz_i=0 
\]
where $\widehat{\cA}={\rm Spf}(\cO_F[[X_1,\dots,X_g]])$ and $F_i(X_1,\dots,X_g)\in \cO_F[[X_1,\dots,X_g]]$.
The final equality holds because $z\in \widehat{\cA}\bigl(p^m\cO^{(1)}\bigr)$, and hence $dz_1 = \cdots = dz_g = 0$. 
\end{proof}

\begin{definition}
We define the \cdef{$w$-topology on $\cB_{\cA}$} to be the projective limit topology with the $w$-topology defined on $\cA(\cO)$. 
\end{definition}
Using \autoref{lemma:GKcontinuousw} and \autoref{lemma:Fontainecontinuousw}, it is a simple exercise, which we leave to the reader, to show that $\varphi_\cA\colon\cB_\cA\to \Lie(\cA)(\cO_F)\otimes_{\cO_F}\C_p(1)$ is $G_F$-equivariant and continuous where $\cB_{\cA}$ has the $w$-topology and $\Lie(\cA)(\cO_F)\otimes_{\cO_F}\C_p(1)$ is endowed with the $p$-adic topology.

\subsection{Comparison isomorphisms via $p$-adic integration}
\label{subsec:comparision}
To conclude this section, we discuss the role of Fontaine integration in the Hodge--Tate and de Rham comparison isomorphisms for abelian varieties.

Fontaine \cite{fontaine:differentials} originally defined this integration map in order to re-prove the Hodge--Tate comparison isomorphism for abelian varieties. More precisely, let $A^\vee$ denote the dual abelian variety, which is also defined over $F$ and with good reduction and we denote by $\cA^\vee$ its N\'eron model over $\cO_F$. See \cite[Theorem 8.4.5]{BLR} for details. 
We write
\begin{align*}
& \varphi_\cA\otimes 1_{\C_p}\colon T_p(A)\otimes_{\Z_p}\C_p\to \Lie(A)(F)\otimes_F\C_p(1) \\
&\varphi_{\cA^\vee}\otimes 1_{\C_p}\colon T_p(A^\vee)\otimes_{\Z_p}\C_p\to {\rm Lie}(A^\vee)(F)\otimes_F\C_p(1),
\end{align*}
which are both surjective maps. 
By duality, the Weil pairing, and further arguments, we obtain the following isomorphism
$$
T_p(A)\otimes_{\Z_p}\C_p\cong \left(\left({\rm Lie}(A^\vee)(F)\right)^\vee\otimes_F\C_p\right)\oplus \left({\rm Lie}(A)(F)\otimes_F\C_p(1)\right),
$$ 
which is the Hodge--Tate decomposition of $T_p(A)$. We also mention work of Coleman \cite{coleman1984hodge}, which complements \cite{fontaine:differentials} and gives another proof of the Hodge--Tate decomposition.

In addition to the Hodge--Tate decomposition, there is a de Rham comparison isomorphism for an abelian variety $A$. More precisely, let $H^1_{\rm dR}(A)$ denotes the first de Rham cohomology group of the abelian variety $A$; it has a natural filtration expressed by the exact sequence:
\[
0\to H^0(A, \Omega^1_{A/F})\to H^1_{\rm dR}(A)\to \Lie(A)(F)\to 0.
\]
In \cite{Colmez1992}, Colmez defined a $p$-adic integration pairing, now called Colmez integration, which is a functorial, perfect pairing
\[
\langle \cdot ,\cdot \  \rangle_{\Cz}\colon T_p(A)\times H^1_{\rm dR}(A)\to B_{\rm dR}^+/I^2,
\]
which realizes the de Rham comparison isomorphism. Here $I\subset B_{\rm dR}^+$ is the maximal ideal of $B_{\rm dR}^+$ and the above pairing is $G_F$-equivariant in the first argument and 
respects filtrations in the second argument. 
By restricting to $H^0(A, \Omega^1_{A/F})\subset H^1_{\rm dR}(A)$, Colmez integration induces a pairing: 
\[
\langle \cdot , \cdot \rangle_{\Fo}\colon T_p(A)\times H^0(A, \Omega^1_{A/F})\to I/I^2\cong \C_p(1),
\]
which gives a map $T_p(A)\to \Lie(A)(F)\otimes_F\C_p(1)$. By \cite[Proposition 6.1]{Colmez1992}, this map coincides with Fontaine's integration $(\varphi_\cA)_{|T_p(A)}$, when $A$ has good reduction. 

%-----------------------------------------------------------------% 
 \section{\bf A $p$-adic uniformization result for the multiplicative group} 
 \label{sec:uniformizationmultiplicativegroup}
 %-----------------------------------------------------------------%
In this section, we study the Fontaine integral for the rigid multiplicative group and deduce a $p$-adic uniformization result for the rigid analytic multiplicative group, which is similar to that of an abelian variety with good reduction. 
We also show that there is a $p$-adic uniformization result for the usual multiplicative group.
% for which, at the moment, we do not have an abelian analogue. 

More precisely, let $F$ denote a finite, unramified extension of $\Q_p$ in $\Kbar$ and denote  by $G:=\mG_m^{\rm rig} := {\rm Spm}(F\langle T,1/T\rangle )$, with the multiplicative group law, the rigid analytic multiplicative group. 
We also consider the usual, algebraic multiplicative group $\mG:=\mG_m:={\rm Spec}\bigl(F[T, 1/T] \bigr)$ with  the multiplicative group law. 
If $L$ is a subfield of $\Kbar$ containing $F$, then we have
$$
G(L)=\bigl(\cO_L^\times, \cdot \bigr)
%\subset \mG(L)=\bigl(L^\times, \cdot \bigr).
$$ 
In particular, $G(\Kbar)=\bigl(\cO^\times, \cdot \bigr) \subset \mG(\Kbar)=\bigl(\Kbar^\times,\cdot \bigr)$. 
In \autoref{defn:Fontaineintegral}, we defined the Fontaine integral for an abelian variety with good reduction, and we point out that the definitions carry over to the rigid analytic multiplicative group. More precisely, we denote by 
$$
\cB_G:=\varprojlim\bigl(\cO^\times\stackrel{\phi}{\leftarrow}\cO^\times\stackrel{\phi}{\leftarrow}\cdots\cO^\times\stackrel{\phi}{\leftarrow}\cdots\bigr) 
$$
and by
$$
\cB_{\mG}:=\varprojlim\bigl(\Kbar^\times\stackrel{\phi}{\leftarrow}\Kbar^\times\stackrel{\phi}{\leftarrow}\cdots\Kbar^\times\stackrel{\phi}{\leftarrow}\cdots\bigr) 
$$
where $\phi(x)=x^p$, and define the group homomorphism $\varphi_G\colon\cB_G\to \C_p(1)$  via 
\[
\varphi_G\bigl((u_n)_{n\ge 0} \bigr):=\pwr{(u_n)^*\pwr{\frac{dT}{T}}}_{n\geq 0} = \pwr{\frac{du_n}{u_n}}_{n\geq 0} \in V_p(\Omega)= \C_p(1)\]
for all $(u_n)_{n\ge 0}\in \cB_G$. 

We remark that the map $\varphi_{\mG}\colon \cB_{\mG}\to \C_p(1)$ is not as obvious as if $x\in \Kbar^\times$, $dx/x$ is not necessarily an element of $\Omega$. But let us remark (we owe this observation to the referee) that if $(v_n)_{n\ge 0}\in 
\cB_{\mG}$, such that $v_0\in \cO\backslash \{0\}$, then there is $n_0\in \N$ such that $p/v_n\in \cO$ for all $n\ge n_0$, 
therefore $pd(v_n)/v_n$
is in $\Omega$ and the sequence 
\[
\left(p\frac{d(v_n)}{v_n} \right)_{n\ge n_0}\in \varprojlim\bigl(\Omega\stackrel{p}{\leftarrow}\Omega\stackrel{p}{\leftarrow}\Omega\stackrel{p}{\leftarrow}\dots \bigr)\subset \C_p(1). 
\]
Using all the notation above, we define the sequence $(a_n)_{n\ge 0}$ where
\[
\begin{array}{ll}
a_n:=p\frac{d(v_n)}{v_n} &  \mbox{ if } n\ge n_0, \text{ and }\\
a_n:=p^{n_0-n+1}\frac{d(v_{n_0})}{v_{n_0}} & \mbox{ for } 0\le n\le n_0.
\end{array}
\]
Then 
\[
(a_n)_{n\geq 0}\in \varprojlim\bigl(\Omega\stackrel{p}{\leftarrow}\Omega\stackrel{p}{\leftarrow}\Omega\stackrel{p}{\leftarrow}\cdots \bigr)\subset \C_p(1).
\]
Finally we define $\varphi_{\mG}\bigl((v_n)_{n\ge 0}\bigr)$ as the image in $\C_p(1)$ of $\displaystyle p^{-1}\bigl(a_n \bigr)_{n\ge 0}$, where $(a_n)_{n\ge 0}$ is the sequence defined above.

If $(v_n)_{n\ge 0}\in \cB_{\mG}$ such that $v_0\in \Kbar^\times\backslash\cO$, we set $\varphi_{\mG}\bigl((v_n)_{n\ge 0}\bigr):=
-\varphi_{\mG}\bigl((v_n^{-1})_{n\ge 0} \bigr)$. 
We observe that if $x\in \cB_G\subset \cB_{\mG}$ then $\varphi_G(x)=\varphi_{\mG}(x)$ as if $x=(u_n)_{n\ge 0}$ then $\displaystyle p(d(u_n)/u_n)=d(u_{n-1})/u_{n-1}$, for all $n\ge 1$. Clearly, the map thus defined, $\varphi_{\mG}\colon \cB_{\mG}\to \C_p(1)$ is a group homomorphism whose restriction to $\cB_G$ is $\varphi_G$.

\begin{remark}
\begin{enumerate}
\item[]
\item Let us first remark that as $\Omega$ is an $\cO$-module we have the canonical isomorphism 
\[\Kbar\otimes_\cO\varprojlim\bigl(\Omega\stackrel{p}{\leftarrow}\Omega\stackrel{p}{\leftarrow}\Omega\stackrel{p}{\leftarrow}\cdots\bigr)\cong \varprojlim\bigl(\Omega\stackrel{p}{\leftarrow}\Omega\stackrel{p}{\leftarrow}\Omega\stackrel{p}{\leftarrow}\cdots\bigr)\otimes_{\Z_p}\Q_p\cong \C_p(1).
\]
Let $(v_n)_{n\ge 0}\in \cB_{\mG}$ such that $v_0\in \cO$. 
Then let us observe that $\displaystyle \varphi_{\mG}\bigl((v_n)_{n\ge 0} \bigr)=v_0^{-1}\Bigl(v_0\frac{d(v_n)}{v_n}\Bigr)_{n\ge 0}$
in $\Kbar\otimes_\cO\varprojlim\bigl(\Omega\stackrel{p}{\leftarrow}\Omega\stackrel{p}{\leftarrow}\Omega\stackrel{p}{\leftarrow}\cdots\bigr)\cong \C_p(1)$. 
\item We denote by $\varphi_{\rm dR}\colon\cB_{\mG}\to B_{\rm dR}^+$ the de Rham integration map for $\mG$ defined by $\varphi_{\rm dR}\bigl((v_n)_{n\ge 0} \bigr):=\log([\tilde{v}]/v_0)$ if $v_0\in \cO$, where $\tilde{v}$ is the image of $\bigl(v_n(\mbox{ mod }p)\bigr)_{n\ge 0}\in \cO_{\C_p}^\flat$. Then, if 
we denote by $I\subset B_{\rm dR}^+$ the maximal ideal of this local ring, we have for $(v_n)_{n\ge 0}\in \cB_{\mG}$ with $v_0\in \cO$: 
$$
\varphi_{\rm dR}\bigl((v_n)_{n\ge 0} \bigr)(\mbox{ mod }I^2)=\frac{[\tilde{v}]-v_0}{v_0}(\mbox{ mod }I^2)=v_0^{-1}\Bigl(v_0\frac{d(v_n)}{v_n}\Bigr)_{n\ge 0}=\varphi_{\mG}\bigl((v_n)_{n\ge 0}\bigr),
$$
under the identification $I/I^2\cong \C_p(1)$. 
\end{enumerate}
\end{remark}

The first result of this section proves the analogue of \autoref{conj:periodic1} for $\mG$, and so implicitly for $G$. 

\begin{theorem}\label{thm:Gm}
Let $\uu = (u_n)_{n\geq 0}\in \cB_{\mG}$, i.e. a sequence such that $u_n\in \cO^\times$ and  $(u_{n+1})^p=u_n$ for all $n\ge 0$. 
Then, $\varphi_{\mG}(\uu) = 0$ if and only if $\uu$ is a periodic sequence which implies that $u_0 \in \mu_m$ with $(m,p) = 1$. 
\end{theorem}

We first prove a lemma.

\begin{lemma}\label{lemma:invariantGM}
Let $\uu:=(u_n)_{n\ge 0}\in \cB_{\mG}$ be an element such that $\alpha(\uu):=u_0\in \cO_L^\times$, for some finite extension $L$ of $F$ in $\Kbar$. Then $\varphi_{\mG}(\uu)=0$ if and only if $\uu\in (\cB_{\mG})^{G_L}$.
\end{lemma}

\begin{proof}
First, we note that $\varphi_{\mG}\colon T_p(\mG)\to \C_p(1)$ is injective and its image is $\mZ_p(1)$ seen as a subgroup of $\C_p(1)$.
Let $\uu\in \cB_{\mG}$ be an element such that $\varphi_{\mG}(\uu)=0$. As for any $\sigma\in G_L$ we have, if $\uu\neq 0$, that $\sigma(\uu)\uu^{-1}\in T_p(\mG)$, we have that $\varphi_\mG(\sigma(\uu)\uu^{-1})=\sigma\bigl(\varphi_\mG(\uu)\bigr)-\varphi_\mG(\uu)=0$. Now using the injectivity of 
the restriction of $\varphi_\mG$ to $T_p(\mG)$ we see that $\sigma(\uu)\uu^{-1}=1$, i.e. ~$\sigma(\uu)=\uu$. 

Conversely, if $\uu=(u_n)_{n\ge 0}\in (\cB_{\mG})^{G_L}$,  then it follows that if denote by $\gamma_L$ a generator of the different ideal of $L$ over $F$, then 
\[
\displaystyle \gamma_L \varphi_{\mG}(\uu)=\left(p\gamma_L \frac{d(u_n)}{u_n} \right)_{n\ge n_0}=0
\]
in $\C_p(1)$. Therefore $\varphi_{\mG}(\uu)=0$.
\end{proof}

\begin{proof}[Proof of \autoref{thm:Gm}]
Let $\uu\in \cB_{\mG}$ be an element with $\varphi_{\mG}(\uu)=0$. 
Let $L$ be a finite extension of $F$ in $\Kbar$ such that
$u_0\in \cO_L$. By the above lemma $\varphi_{\mG}(\uu)=0$ is equivalent to $\uu\in (\cB_{\mG})^{G_L}$ i.e. $u_n\in \cO_L$ for every 
$n\ge 0$. Fix $m\ge 0$. Then $u_m$ is a $p^r$-th power in $L$ (i.e we have  $u_m=u_{m+r}^{p^r}$) for every $r\ge 0$, so by \cite[Chapter 1, Section 7, Proposition]{Fesenko:LocalFields} this is equivalent to the fact that $u_m$ is a $q-1$-root of $1$ in $L$, where $q$ is the cardinal of the residue field of $L$. Therefore $\varphi_{\mG}(\uu)=0$ is equivalent to $\uu$ being a periodic sequence.
\end{proof}

With this observation, we now move onto our uniformization result. 
Let us recall that we denoted $\mG:=\mG_m$, and let $\mG^w(\Kbar)$ denote $\mG(\Kbar) = \Kbar^{\times}$ with the linear $w$-topology. This topology has a base of open neighbourhoods of $1$ given by
$\{1+p^n\cO^{(1)}\}_{n\in \N}$ and for every $x\in \mG(\Kbar)=\Kbar^\times$, a base of open neighbourhoods of $x$ is given by
$\{x(1+p^n\cO^{(1)})\}_{n\in \N}$. We observe that the sugbroup $G(\Kbar)$, with the induced topology becomes an open subgroup, which will be denoted $G^w(\Kbar)$.

First, we have continuity of the Fontaine integral. 

\begin{lemma}\label{lemma:continuousGM}
The map $\psi\colon G^w(\Kbar) \to \Omega$ given by $\psi(x)=dx/x\in \Omega$ is continuous for the discrete topology on $\Omega$. 
\end{lemma}

\begin{proof}
Let  $\omega\in \Omega$ and let $U:=\psi^{-1}(\omega) \subset G^w(\Kbar)$. We want to show that $U$ is open.
Suppose $y\in U$, i.e. ~$y\in G^w(\Kbar)$ and $dy/y=\omega$. Let $n\in \N$ and 
$u\in y(1+p^n\cO^{(1)})$, i.e. $u=y(1+p^nv)$, with $v\in \cO^{(1)}$. We have $du/u=dy/y+d(1+p^nv)/(1+p^nv)$.
Moreover, we see that 
\[
\frac{d(1+p^nv)}{1+p^nv}=\frac{d(p^nv)}{1+p^nv}=\frac{p^ndv}{1+p^nv}=0,
\]
and so $du/u=dy/y=\omega$, i.e. ~$y(1+p^n\cO^{(1)})\subset U$. This implies that $U$ is open. 
\end{proof}

\begin{lemma}
\label{lemma:tortop}
Let  $\mG_{\rm tor}=\mG_{p{\text{-tors}}}\oplus \mG_{p'{\text{-tors}}}\subset \mG(\Kbar)$ denote the torsion subgroup of $\mG(\Kbar)$.
Then the $w$-topology of $\mG(\Kbar)$ induces the discrete topology on $\mG_{\rm tor}$.
\end{lemma}

\begin{proof}
As $G^w(\Kbar)$ is open in $\mG^w(\Kbar)$ and $\mG(\Kbar)_{\rm tor}=G(\Kbar)_{\rm tor}$ it is enough to work with
$G(\Kbar)$.
Let $\zeta\in \mu_{p^\infty}(\Kbar)$ and let  $x\in \zeta(1+p^n\cO^{(1)})\cap \mu_{p^\infty}(\Kbar)$.
Then $y:=x\zeta^{-1}\in 1+p^n\cO^{(1)}$, so $dy=0$, i.e. $y\in \cO^{(1)}\cap \mu_{p^\infty}(\Kbar)=\{1\}$.
So $x=\zeta$. On the other hand $G_{p'{\text{-tors}}}\subset \cO_K^\times\subset \bigl(\cO^{(1)}\bigr)^\times\subset \cO^\times$ and the $w$-topology
of $\cO^\times$ induces the $p$-adic toplogy of $\cO_K^\times$, which induces the discrete topology 
on $G_{p'{\text{-tors}}}$. 
\end{proof}

The $w$-topology on $\mG(\Kbar)$ defines the projective limit topology on $\cB_{\mG}$ such that the sequence 
\begin{equation}\label{eqn:GMSES}
1\to T_p\left(\mG^w(\Kbar)\right)\cong\Z_p(1)\to \cB_{\mG}\to \mG^w(\Kbar)\to 1
\end{equation}
is an exact sequence of topological abelian groups.
By \autoref{lemma:tortop}, the topology induced by $\cB_{\mG}^w$ on $\Z_p(1)$ is the $p$-adic topology, and it follows from \autoref{lemma:continuousGM} that $\varphi_{\mG}\colon\cB_{\mG}\to \C_p(1)$ is continuous, with the $p$-adic topology on $\C_p(1):=V_p(\Omega)$ as for group homomorphisms proving continuity reduces to showing that the inverse image of a neighborhood of 0 is a neighborhood of 1. Moreover, as the $G_K$-action on $\Kbar^w$ is continuous (cf.~\cite{iovita_zaharescu}) it follows that the $G_K$-action on $\mG^w(\Kbar)=(\Kbar^\times)^w$ is continuous, and since the $G_K$-action on $\cB_{\mG}$ is continuous, the exact sequence \eqref{eqn:GMSES}  is an exact sequence of continuous $G_K$-representations. 
\autoref{thm:Gm} asserts that $\ker(\varphi_\mG) = \per(\cB_G)$, the periodic sequences in $\cB_\mG$ (cf.~\autoref{defn:periodicpaths}). 
Also, under the identification $T_p(\Omega)\otimes_{\Z_p}\Q_p\cong \C_p(1)$, $\varphi_{\mG}(\Z_p(1))$ is identified with $\Z_p(1)\subset \C_p(1)$. Therefore we have the following commutative diagram with exact rows and columns:
\[
\begin{tikzcd}[column sep = .4cm, row sep = .4cm]
{} & {} & 0 \arrow{d} & 0 \arrow{d} & {} \\
{} & {} & \per(\cB_\mG) \arrow{d} \arrow[equal]{r} & \mG_{p'-\text{tor}} \arrow{d} & {} \\
0 \arrow{r} & \mZ_p(1) \arrow{r}\arrow[equal]{d} & \cB_{\mG} \arrow{r} \arrow{d}{\varphi_G} & \mG^w(\Kbar) \arrow{r} \arrow{d} &  0 \\
0 \arrow{r} & \mZ_p(1) \arrow{r}& \mC_p(1) \arrow{r}  &\mC_p(1)/\mZ_p(1)\arrow{r} &  0
\end{tikzcd}
\]
where the equal signs mean isomorphisms. We have a similar diagram for $G$,  
in particular we have injective, continuous and $G_K$-equivariant maps 
$$
\begin{array}{ccccccccc}
&\iota_G\colon& G(\Kbar)^{(p),w}:=G^w(\Kbar)/G_{p'{\text{-tors}}}&\hookrightarrow &\C_p(1)/\Z_p(1)\\
&&\cap&&||\\
&\iota_\mG\colon& \mG(\Kbar)^{(p),w}:=\mG^w(\Kbar)/\mG_{p'{\text{-tors}}}&\hookrightarrow &\C_p(1)/\Z_p(1).
\end{array}
$$

For later use as well, we will consider a more general situation:~let $T$ be a free $\Z_p$-module of rank $m\ge 1$ with a continuous $G_F$-action, $V$ a free $\cO_F$-module of rank $s\ge 1$, where $G_F$ acts trivially on $V$, and an injective $\Z_p$-linear, $G_F$-equivariant map $\varphi\colon T\hookrightarrow V\otimes_{\cO_F}\C_p(n)$, with $n\in \Z$, $n\neq 0$.
We remark that the existence of $\varphi$ implies $T^{G_K}=0$. 

We now define a certain class of elements in $(V\otimes_{\cO_F}\C_p(n))/\varphi(T)$.

\begin{definition}
Let $x\in (V\otimes_{\cO_K}\C_p(n))/\varphi(T)$. 
We say $x$ is \cdef{algebraic} if the orbit $G_F\cdot x\subset (V\otimes_{\cO_F}\C_p(n))/\varphi(T)$ is finite. If $x$ is algebraic, then there is $L\subset \Kbar$, $[L:F]<\infty$ such that $x\in \left((V\otimes_{\cO_F}\C_p(n))/\varphi(T)\right)^{G_L}$. 
\end{definition}

\begin{definition}\label{defn:cyrstalline}
Suppose that $T$ is a crystalline $G_F$-representation, i.e. $T\otimes_{\Z_p}\Q_p$ is a crystalline $\Q_p$-representation of $G_F$. 
Let $x\in (V\otimes_{\cO_F}\C_p(n))/\varphi(T)$ be an algebraic element. 
\begin{enumerate}
\item We say that $x$ is a \cdef{crystalline} element 
if one of the following happens:
\begin{enumerate}
\item $x$ is a torsion element, or
\item  $x$ is non-torsion, and if we denote 
\[
\alpha\colon V\otimes_{\cO_F}\C_p(n)\to (V\otimes_{\cO_F}\C_p(n))/\varphi(T)
\] 
the natural projection, then we have that $\alpha^{-1}(x\Z_p) \otimes_{\mZ_p} \mQ_p$ is a crystalline $G_L$-representation, for every $L$ with $[L:F]<\infty$ such that $x\in 
\left((V\otimes_{\cO_K}\C_p(n))/\varphi(T)\right)^{G_L}$.
\end{enumerate}
\item We say that $x$ is \cdef{semi-stable} if one of the following happens:
\begin{enumerate}
\item $x$ is a torsion element, or
\item  $x$ is non-torsion, and 
 $\alpha^{-1}(x\Z_p) \otimes_{\mZ_p} \mQ_p$ is a semi-stable $G_L$-representation, for every $L$ with $[L:F]<\infty$ such that $x\in 
\left((V\otimes_{\cO_K}\C_p(n))/\varphi(T)\right)^{G_L}$.
\end{enumerate}
\end{enumerate}
\end{definition}

\begin{remark}\label{rem:crystalline}
Let us observe that if $T$ is a crystalline $G_F$ representation and $x\in \left(V\otimes_{\cO_F}\C_p(n)\right)/\varphi(T)$ is a non-torsion, algebraic element, invariant by $G_L$ for $[L:F]<\infty$, we have a natural commutative diagram with exact rows
\[
\begin{tikzcd}[row sep = 1.2em]
0\arrow{r} & T \arrow{r}{\varphi} \arrow[equal]{d}& V\otimes_{\cO_F}\C_p(n)\arrow{r}{\alpha} & \left(V\otimes_{\cO_F}\C_p(n)\right)/\varphi(T) \arrow{r} & 0\\
0\arrow{r} &T \arrow{r}{\varphi}& \alpha^{-1}(x\Z_p) \arrow{r}{\alpha} \arrow[right hook->]{u} &x \Z_p \arrow{r} \arrow[right hook->]{u}&0.
\end{tikzcd}
\]
Therefore, $\alpha^{-1}(x\Z_p)\otimes_{\Z_p}\Q_p$  is a $\Q_p$-vector space of rank ${\rm dim}_{\Z_p}(T)+1$ which is a $G_L$-representation. This representation is crystalline if and only if the Galois cohomology class defined by the second row of the diagram tensored with $\Q_p$ over $\Z_p$ lives in ${\rm H}^1_f(L, T\otimes_{\Z_p}\Q_p)$. 
\end{remark}

Now we come back to the multiplicative and respectively the rigid multiplicative group $\mG$ and respectively $G$.
In the notations above, in both cases $T:=T_p(\mG)=T_p(G)=\Z_p(1)$ and $V=\cO_F$ 

\begin{theorem}
We have that an element $x\in \C_p(1)/\Z_p(1)$ belongs to the image of $\iota_G$ if and only if $x$ is crystalline and it belongs to the image of $\iota_{\mG}$ if and only if it is semi-stable.
\end{theorem}

\begin{proof}
We remark that the torsion subgroup of $\C_p(1)/Z_p(1)$ is the subgroup 
\[
\Q_p(1)/\Z_p(1)\cong \iota_G(G[p^{\infty}](\Kbar)=\mu_{p^\infty}(\Kbar)=\iota_{\mG}\bigl(\mG[p^\infty](\Kbar)\bigr).
\] 
Therefore
$\iota_G$ and respectively $\iota_{\mG}$ induce isomorphisms on the torsion subgroups of the domain and target, so we may look at the map induced by $\iota_G$, respectively $\iota_{\mG}$ on the quotients of the domain and target by the torsion subgroups i.e., we look at $\iota_G\colon\cO^\times\otimes_\Z\Q\hookrightarrow \C_p(1)/\Q_p(1)$ and $\iota_{\mG}\colon \Kbar^\times\otimes_{\Z}\Q\hookrightarrow \C_p(1)/\Q_p(1)$. 

Let $x\in \C_p(1)/\Q_p(1)$ and suppose it is algebraic, i.e. invariant under some $G_L$, for $L$ a finite extension of $F$ in $\Kbar$. Consider the exact sequence in the context of the diagram of $p$-adic $G_L$-representations
\[
\begin{tikzcd}[column sep = .4cm, row sep = .4cm]
0 \arrow{r} & V_p(G) \arrow{r} \arrow[equal]{d}{\varphi_G} & \cB_{G,L} \otimes_{\mZ} \mQ \arrow{r} \arrow{d}{\varphi_G} & \cO_L^{\times} \otimes_{\mZ} \mQ \arrow{r} \arrow{d}{\varphi_G} & 0\\
0 \arrow{r} & \mQ_p(1) \arrow{r} & \mC_p(1) \arrow{r} & \mC_p(1)/\mQ_p(1) \arrow{r} & 0.
\end{tikzcd}
\]
The $G_L$-continuous cohomology diagram associated to it is
$$
\begin{array}{ccccccccccc}
0&\longrightarrow&\cO_L^\times\otimes_\Z\Q&\stackrel{\partial}{\longrightarrow}&H^1(G_L, \Q_p(1))&\longrightarrow\\
&&\cap\iota_G&&||\\
0&\longrightarrow&\bigl(\C_p(1)/\Q_p(1)\bigr)^{G_L}&\stackrel{\partial'}{\longrightarrow}&
H^1(G_L, \Q_p(1))&\longrightarrow&0
\end{array}
$$
By \cite[Example 3.10.1 \& Proposition 5.4]{bloch_kato}, the image of $\partial$ is $H_f^1(G_L, \Q_p(1))\subset H^1(G_L, \Q_p(1))$,  and so we see that $x\in \bigl(\C_p(1)/\Q_p(1) \bigr)^{G_L}$ is in the image of $\iota_G$ if and only if $\partial'(x)\in H^1_f(G_L, \Q_p(1))$, i.e. ~if and only if $x$ is crystalline.

Similarly, we have the commutative diagram with exact rows
\[
\begin{tikzcd}[column sep = .4cm, row sep = .4cm]
0 \arrow{r} & V_p(\mG) \arrow{r} \arrow[equal]{d}{\varphi_\mG} & \cB_{\mG,L} \otimes_{\mZ} \mQ \arrow{r} \arrow{d}{\varphi_\mG} & L^{\times} \otimes_{\mZ} \mQ \arrow{r} \arrow{d}{\varphi_\mG} & 0\\
0 \arrow{r} & \mQ_p(1) \arrow{r} & \mC_p(1) \arrow{r} & \mC_p(1)/\mQ_p(1) \arrow{r} & 0.
\end{tikzcd}
\]

The $G_L$-continuous cohomology diagram associated to it is
$$
\begin{array}{ccccccccccc}
0&\longrightarrow&L^\times\otimes_\Z\Q&\stackrel{\partial_{\mG}}{\longrightarrow}&H^1(G_L, \Q_p(1))&\longrightarrow\\
&&\cap\iota_\mG&&||\\
0&\longrightarrow&\bigl(\C_p(1)/\Q_p(1)\bigr)^{G_L}&\stackrel{\partial'}{\longrightarrow}&
H^1(G_L, \Q_p(1))&\longrightarrow&0
\end{array}
$$
By Kummer theory and Bloch--Kato \cite{bloch_kato},  the image of $\partial_{\mG}$ is ${\rm H}^1(L, \Q_p(1))={\rm H}^1_{\rm st}(L, \Q_p(1))$.
Therefore an element $x\in \bigl(\C_p(1)/\Z_p(1)\bigr)^{G_L}$ is in the image of $\iota_{\mG}$ if and only if it is semi-stable.
\end{proof}

\begin{remark}
We observe that the sequence $\bigl(\iota_{\mG}(1+p^n\zeta_{2n})\bigr)_{n\ge 0}$ is unbounded in $\C_p(1)/\Z_p(1)$, therefore 
$\iota_{\mG}$ is not continuous for the $p$-adic topology on $\mG(\Kbar)$ i.e. the topology having a base of neighbourhoods of $1$ given by $\bigl(1+p^n\cO \bigr)_{n\ge 0}$. Therefore $\iota_{\mG}$ does not extend to $\mG(\C_p)=\mC_p^\times$. 
 \end{remark}

%-----------------------------------------------------------------%
\section{\bf The result on the zeroes of the Fontaine integral}
\label{sec:conjecture}
 %-----------------------------------------------------------------%
In this section, we return to the study of an abelian variety $A$ over the finite, unramified extension $F$ of $\Q_p$, with good reduction.
Recall that the Fontaine integral $\varphi_\cA\colon\cB_\cA\to {\rm Lie}(A)(F)\otimes_F\C_p(1)$ will vanish on an element $\underline{u} = (u_n)_{n\geq 0} \in \cB_{\cA}$ if $u_n \in \cO_K$ for all $n$ i.e., if the path $\underline{u}$ lies in $(\cB_{\cA})^{G_K}$. 
It is natural to speculate whether the Fontaine integral restricted to $T_p(A)$ will be injective away from unramified paths which live in $T_p(A)$. 
To make this question easier to study, we will impose the following assumption throughout. 

\begin{assumption}\label{assumption}
We assume that the abelian variety $A$ defined over $F$ satisfies $T_p(A)^{G_K}=0$. 
\end{assumption}

\begin{remark}
\autoref{assumption} is not very restrictive. 
For example, if $A$ is an elliptic curve, \autoref{assumption} is equivalent to the property that $A$ does not have {\rm CM} by a quadratic imaginary field $M$, in which $p$ is split (see e.g., \cite[Theorem 2.11]{ozeki2010torsion} and \cite[A.2.4]{serreabelianl-adic} for details). 
We refer the reader to \cite[A.2.3]{serreabelianl-adic} for further discussion of the relationship between complex multiplication and \autoref{assumption}. 
\end{remark}

We begin our analysis of $\ker(\varphi_{\cA})$ by identifying two of its subgroups. 
First, we again remark that $(\cB_\cA)^{G_K}\subset \ker(\varphi_\cA)$. Indeed, this follows from the above discussion, or one may deduce it from the fact that 
\[
\left(\Lie(\cA)(\cO_F)\otimes_{\cO_F}\C_p(1)\right)^{G_K}=0,
\]
which is a consequence of a result of Tate \cite[Theorem 2]{tate:pdiv}. We can easily determine another significant subgroup of $\ker(\varphi_\cA)$.

\begin{definition}\label{defn:periodicpaths}
A path $\uu = (u_n)_{n\geq 0} \in \cB_{\cA}$ is \cdef{periodic} if there exists some $k\geq 1$ such that $u_n = u_{n+k}$ for all $n\ge 0$. More precisely, a periodic path is of the form
\begin{equation}\label{eqn:periodicsequence}
\underline{u}:= \left(u_0, [p^{k-1}]u_0, [p^{k-2}]u_0\dots,[p]u_0, u_0, [p^{k-1}]u_0,\dots,[p]u_0, u_0, [p^{k-1}]u_0\dots\right), 
\end{equation}
where, let us remark that $[p^k]u_0=u_0$.
We define $\per(\cB_\cA) \subset \left(\cB_\cA\right)^{G_K}$ to be the subgroup of periodic paths of $\cB_{\cA}$. 
\end{definition}

Periodic paths enjoy the following properties.

\begin{lemma}\label{lemma:periodicprimetop}
\begin{enumerate}
\item []
\item Let $\uu = (u_n)_{n\geq 0} \in \cB_{\cA}$ be a periodic path. Then $u_0$ is a prime-to-$p$ torsion point of $\cA(\cO)$ and $\uu \in (\cB_{\cA})^{G_K}$. 
\item Suppose that $A$ satisfies \autoref{assumption}, and let $u$ be a prime-to-$p$ torsion point of $\cA$. 
Then, there is a unique periodic path $\uu = (u_n)_{n\geq 0} \in (\cB_{\cA})^{G_K}$ such that $u_0= u$.  
\item The subgroup $\per(\cB_\cA)$ is isomorphic to the subgroup of prime-to-p torsion points on $\cB_{\cA}$ and the map $\uu = (u_n)_{n\geq 0} \mapsto u_0$ gives an isomorphism of $\per(\cB_\cA)$ to the  subgroup of prime-to-p torsion points on $\cA(\cO)$. 
\end{enumerate}
\end{lemma}

\begin{proof}
After recalling the well-known fact that a prime-to-$p$ torsion point of $\cA(\cO)$ is unramified i.e., is defined over $K$, the proofs of these statements are exercises, which we leave to the reader.
\end{proof}

With these definitions established, we can now state our result on the vanishing of the Fontaine integral.

\begin{theorem}[$=$\autoref{conj:periodic1}]\label{conj:periodic}
Let $A$ be an abelian variety over $F$, with good reduction, and let $\cA$ denote its N\'eron model. 
Suppose that $A$ satisfies \autoref{assumption}. Then we have that 
\[\per(\cB_\cA)=\left(\cB_\cA\right)^{G_K}=\ker(\varphi_\cA).\]
\end{theorem}

Below, we make a remark relating \autoref{conj:periodic} to previous literature on determining the zeros of $p$-adic integrals. 

\begin{remark}
We wish to highlight the similarities between our \autoref{conj:periodic} and \cite[Theorem 2.11]{coleman1985torsion}. 
The result of Coleman says that the torsion points on an abelian variety correspond to the set of common zeros of $p$-adic abelian integrals of the first kind on the abelian variety. 
Our \autoref{conj:periodic} shows that for an abelian variety $A$ over $K$ satisfying \autoref{assumption}, the common zeros of the Fontaine integral are precisely the torsion points on $\cB_{\cA}$. 
Indeed, recall that \autoref{lemma:periodicprimetop}.3 shows the subgroup $\per(\cB_{\cA})$ coincides with the prime-to-$p$ torsion on $\cB_{\cA}$, and since $\cB_{\cA}$ will have no $p$-power torsion by construction, the statement follows. 
\end{remark}

%-----------------------------------------------------------------%
\section{\bf The proof of \autoref{conj:periodic}}
\label{sec:evidence}
%-----------------------------------------------------------------%
In this section, we first describe the reduction of the proof of \autoref{conj:periodic} to a statement about injectivity of the Fontaine integral restricted to the Tate module of the formal group associated to our abelian variety and then  prove this injectivity following an idea of P.~Colmez. 
 
\subsection{Reduction of \autoref{conj:periodic} to the setting of the Tate module of the formal group}
\label{sec:reductiontoformal}
Let $A/F$ be an abelian variety with good reduction, $\cA$ its N\'eron model over $\cO_F$ and we assume that
$A$ satisfies \autoref{assumption}. 
Let $\widehat{\cA}$ denote the formal completion of $\cA$ along the identity of its special fiber, i.e. the formal group of $\cA$. Then, if we denote by $g$ the dimension of $A$ over ${\rm Spec}(F)$, we have that 
\[
\widehat{\cA}\cong{\rm Spf}\left( \cO_F[[X_1,\dots,X_g]]\right)
\] 
as formal schemes, with a formal group-law given by $g$ power series 
\[
\widehat{\cA}_1(X_1,\dots,X_g,Y_1,\dots,Y_g),\dots,\widehat{\cA}_g(X_1,\dots,X_g,Y_1,\dots,Y_g).
\]
Let $T_p(\widehat{\cA})\subset T_p(A)$ be the Tate module of the formal group and let us recall that the restriction to the formal group defines an isomorphism $H^0(\cA, \Omega^1 _{\cA/\cO_F})\cong {\rm Inv}(\widehat{\cA})$, where the second $\cO_F$-module is the module of invariant differentials of the formal group. We also have natural isomorphisms $\Lie(\cA)(\cO_F)\cong {\rm Lie}(\widehat{\cA})(\cO_F)$, where the second module is the Lie algebra of the formal group, i.e. its tangent space at the origin.

Let $(\varphi_\cA)_{|T_p(\widehat{\cA})}\colon T_p(\widehat{\cA})\to {\rm Lie}(\widehat{\cA})(\cO_F)\otimes_{\cO_F}\C_p(1)$ denote the restriction of 
$\varphi_\cA$ to $T_p(\widehat{\cA})$.
The main result of this subsection proves that if $(\varphi_\cA)_{|T_p(\widehat{\cA})}$ is injective, then \autoref{conj:periodic} holds.

\begin{theorem}\label{thm:reductiontoformal}
Assume that $A$ satisfies \autoref{assumption} and is such that $\ker((\varphi_\cA)_{|T_p(\widehat{\cA})})=0$. Then we have
\begin{enumerate}
\item The restriction $(\varphi_\cA)_{|T_p(A)}\colon T_p(A)\to \Lie(\cA)(\cO_F)\otimes_{\cO_F}\C_p(1)$ is injective.
\item We have $\ker(\varphi_\cA)=\per(\cB_\cA)=\left(\cB_\cA\right)^{G_K}$ i.e., \autoref{conj:periodic} holds. 
\end{enumerate}
\end{theorem}
 
\begin{proof}[Proof of \autoref{thm:reductiontoformal}.(1)]
We first prove $(\varphi_\cA)_{|T_p(A)}$ is injective. 
If we denote by $\cA(p)$ and respectively $\widehat{\cA}(p)$, the $p$-divisible groups attached to the abelian scheme $\cA$ and respectively its formal group scheme, then we have an exact sequence of $p$-divisible groups
\[
0\to \widehat{\cA}(p)\to \cA(p)\to \cA(p)/\widehat{\cA}(p)\to 0.
\] 
As $\widehat{\cA}(p)$ is identified with the connected sub-$p$-divisible group of $\cA(p)$, the $p$-divisible group $\cA(p)/\widehat{\cA}(p)$
is an \'etale $p$-divisible group over $\cO_K$, therefore it is isomorphic, as $p$-divisible groups, to $\left(\Q_p/\Z_p \right)^{2g-h}$, where $h$ is the height of $\widehat{\cA}$. The exact sequence of $p$-divisible groups above defines and exact sequence of $G_K$-representations 
\begin{equation}\label{eqn:SESTatemodules}
0\to T_p(\widehat{\cA})\to T_p(A) \stackrel{\psi}{\to} \left(\Z_p\right)^{2g-h}\to 0,
\end{equation}
The exact sequence \eqref{eqn:SESTatemodules} gives the long exact sequence in cohomology:
\[
0 = T_p(A)^{G_K}\to \left(\Z_p\right)^{2g-h}\stackrel{\delta}{\to}H^1(G_K, T_p(\widehat{\cA}))
\]
where $T_p(A)^{G_K}=0$ by \autoref{assumption}. 
If $(\varphi_\cA)_{|T_p(A)}(x) = 0$, then $(\varphi_\cA)_{|T_p(A)}((\sigma - 1)x) = 0$ for all $\sigma\in G_K$. 
However, $(\sigma - 1)x$ lies in $T_p(\widehat{\cA})$ and the assumption then implies that $(\sigma - 1)x = 0$ for all $\sigma \in G_K$. In particular, $\delta(\psi(x)) = 0$ by definition of the connecting map.  
As $\delta$ is injective, it follows that $\psi(x)=0$, i.e. $x\in T_p(\widehat{\cA})$. 
By our assumption that  $\ker((\varphi_\cA)_{|T_p(\widehat{\cA})})=0$, we have that $x=0$.  
\end{proof}

Before we prove \autoref{thm:reductiontoformal}.(2), we need a lemma. 

\begin{lemma}\label{lemma:invariants}
Keep the notations and assumptions from \autoref{thm:reductiontoformal}. 
Let $\uu\in \cB_\cA$ be such that $u_0=\alpha(\uu)\in \cA(\cO_L)$, for $L/F$ a finite extension in $\overline{K}$. 
If $\varphi_\cA(\uu)=0$, then $\uu\in (\cB_\cA)^{G_L}$, where $G_L={\rm Gal}(\overline{K}/L)$.
\end{lemma}

\begin{proof}
First note that for every $\sigma\in G_L$ we have that $\alpha\left(\sigma(\uu)-\uu\right)=0$, i.e. $\sigma(\uu)-\uu\in T_p(\cA)$.
\autoref{thm:reductiontoformal}.(1) implies that $(\varphi_{\cA})_{|T_p({\cA})}$ is injective. Therefore if $(\varphi_\cA)(\uu)=0$, then we have $(\varphi_{\cA})_{|T_p({\cA})}(\uu-\sigma(\uu))=\varphi_\cA(\uu)-\sigma(\varphi_{\cA}(\uu))=0$, and hence $\sigma(\uu)-\uu=0$ for all $\sigma\in G_L$. Thus, for all $\sigma\in G_L$ we have $\sigma(\uu)=\uu$.
\end{proof}

\begin{proof}[Proof of \autoref{thm:reductiontoformal}.(2)]
We already know that $\per(\cB_\cA)\subset \ker(\varphi_\cA)$, so it suffices to prove the reverse inclusion. 
We break the proof up into three cases. \\

\item \textit{Case 1}.~Let $\underline{u}\in \ker(\varphi_\cA)$ such that $u_0:=\alpha(\underline{u})\in \cA(\cO)[m]$ for $m\ge 1$ with $(m,p)=1$.
Then $u_0\in \cA(\cO_K)$ and by \autoref{lemma:invariants}, we see that $\underline{u}\in (\cB_\cA)^{G_K}$. 
However, the only element $\uu\in (\cB_\cA)^{G_K}$ with $u_0\in \cA(\cO_K)[m]$ for $m$ above is the periodic sequence associated to $u_0$, and so $\uu\in \per(\cB_\cA)$.\\

\item \textit{Case 2}.~Let us now suppose that $\uu\in \ker(\varphi_\cA)$ with $u_0=\alpha(\uu)\in \cA(\cO)$ such that $u_0$ is a torsion point of $\cA$ of order a power of $p$, say $u_0\in \cA(\cO)[p^a]$ for some $a\geq 1$. 
Then $p^a(\uu)=\left([p^a]u_n\right)_{n\ge 0}\in T_p(A)$ and $\varphi_\cA(p^a(\uu))=p^a\varphi_\cA(\uu)=0$. 
As $(\varphi_\cA)_{|{T_p(A)}}$ is injective, we have that $p^a(\uu)=0$ in $\cB_\cA$. Since $\cB_\cA[p]=0$, we see that $\uu=0$.\\

\item \textit{Case 3}.~Finally, let us suppose that $\uu\in \ker(\varphi_\cA)$ with $u_0=\alpha(\uu)$ not a torsion point of $\cA(\cO)$. 
Observe that $u\otimes 1_{\Q_p}\in \cB_\cA\otimes_{\Z} \Q $ whose image via $\alpha\otimes 1_{\Q}$ in $\cA(\cO)\otimes_{\Z} \Q$ is $u_0\otimes 1\neq 0$. We have a natural exact sequence of $\Q_p$-vector spaces, with continuous action by $G_K$, namely
\begin{equation}\label{eqn:SESrational}
0\to V_p(A)\to \cB_\cA\otimes_{\Z} \Q\to \cA(\cO)\otimes_{\Z} \Q\to 0.
\end{equation}
Notice that as $T_p(A)$ is naturally a $\Z_p$-module, we have that $V_p(A):=T_p(A)\otimes_{\Z}\Q\cong T_p(A)\otimes_{\Z_p}\Q_p$.

Let us now suppose that $\uu\in \ker(\varphi_\cA)$ has the property that $u_0=\alpha(\uu)\in \cA(\cO_L)$ for some finite extension $L$
of $F$, in $\overline{K}$ and that, as stated above, $0\neq u_0\otimes 1\in \cA(\cO)\otimes_{\Z} \Q$. We denote, as usual $G_L={\rm Gal}(\overline{K}/L)$ and let $\alpha_L\colon\cB_{\cA, L}\to \cA(\cO_L)$ be the fiber product of the diagram of $G_L$-modules:
$$
\begin{array}{cccccccccc}
\cB_{\cA}&\stackrel{\alpha}{\longrightarrow}&\cA(\cO)\\
&&\cup\\
&&\cA(\cO_L)
\end{array}
$$
In other words we have a cartesian, commutative diagram of $G_L$-modules, with exact rows 
$$
\begin{array}{ccccccccccc}
0&\to& V_p(A)&\to& \cB_\cA\otimes_{\Z} \Q&\stackrel{\alpha\otimes 1}{\to}& \cA(\cO)\otimes_{\Z} \Q&\to& 0\\
&&||&&\cup&&\cup\\
 0& \to& V_p(A)&\to& \cB_{\cA,L}\otimes_{\Z} \Q&\stackrel{\alpha_L\otimes 1}{\to}& \cA(\cO_L)\otimes_{\Z} \Q&\to& 0
\end{array}
$$
It is easy to see that we have, on the one hand
$$
(\cB_\cA)^{G_L}\subset \cB_{\cA,L}\subset\cB_\cA,
$$
therefore we have $\bigl(\cB_{\cA,L}\otimes_{\Z}\Q \bigr)^{G_L}=\bigl(\cB_\cA\otimes_\Z\Q \bigr)^{G_L}$.
On the other hand the exact sequence 
\[
0\to V_p(A)\to \cB_{\cA,L}\otimes_{\Z} \Q\to \cA(\cO_L)\otimes_{\Z} \Q\to 0
\]
is an exact sequence of finite $\Q_p$-vector spaces with continuous $G_L$-action, whose long exact $G_L$-cohomology sequence reads as
\[
\left(\cB_\cA \otimes_{\Z} \Q\right)^{G_L}\stackrel{\alpha_L\otimes 1}{\longrightarrow} \left(\cA(\cO_L)\otimes_{\Z}\Q\right)\stackrel{\partial}{\longrightarrow} H^1(G_L, V_p(A)).
\]
By \cite[Example 3.11]{bloch_kato}, the map $\partial$ is injective, and therefore, we have that 
\[\alpha_L\otimes 1\left(\left(\cB_\cA\otimes_{\Z} \Q\right)^{G_L}\right) = 0.\] 
Thus, there is no $G_L$-invariant element of $\cB_\cA\otimes \Q$ which has  image under $\alpha\otimes 1$ equal to $u_0\otimes 1_{\Q}\neq 0$, and so $\uu$ cannot be $G_L$-invariant. However, we note that this contradicts \autoref{lemma:invariants}. Therefore, we can conclude and say that $\ker(\varphi_\cA)=\per(\cB_\cA)$, as desired. 
\end{proof}

\subsection{Injectivity of Fontaine integral restricted to Tate module of formal group}
\label{sec:classes}
We now present a proof of the injectivity of Fontaine's integration map, restricted to the Tate module of the formal group of the abelian variety. The outline of this proof was communicated to us by Pierre Colmez.

We briefly recall the relevant notation. 
Let $\widehat{\cA}$ be the formal group of an abelian variety $A$, with good reduction over a finite extension $F$ of $\Q_p$, in $K$, where we recall that
$K$ is the maximal unramified extension of $\Q_p$ inside some algebraic closure $\overline{\Q}_p$. 
In this subsection, we {\bf do not} assume \autoref{assumption}, i.e.~we \textbf{do not} assume that that $T_p(A)^{G_K}=0$.

Let $D:=H^1_{\rm dR}(\widehat{\cA})$ denote the Dieudonn\'e module of $\widehat{\cA}$ over $F$, which is an admissible filtered, Frobenius module over $F$. 
Moreover, $D\cong D_{\rm cris}(V)$, for $V:=V_p(\widehat{\cA})^\vee$, by which we mean the $\Q_p$-dual of the $G_F$-representation $V_p(\widehat{\cA})$. 
Let $\phi_D$ denote the linear Frobenius of $D$, i.e. ~if $[F:\Q_p]=r$ then $\phi_D$ is the linear map $(\phi^{\rm abs}_D)^r$ , where 
$\phi^{\rm abs}_D$ is the absolute Frobenius on $D$. We denote by $W\subset D$ the subspace which determines the filtration of $D$, that is:~${\rm Fil}^i(D)=D$ 
for $i\le 0$, ${\rm Fil}^1(D)=W$ and ${\rm Fil}^j(D)=0$ for $j\ge 2$. 
We note that $W$ is the subspace of invariant differentials. 

For every filtered, Frobenius submodule $D'$ of $ D$, we define the \cdef{Hodge and Newton numbers} of $D'$, $t_H(D')$, respectively $t_N(D')$ in terms of the filtration and the dimensions of its graded quotients and respectively, in terms of the normalized slopes of the Frobenius $\phi_{D'}$.
We say that $D$ is \cdef{weakly admissible} if for every $D'\subset D$ sub-filtered, Frobenius module we have $t_N(D')\ge t_H(D')$, with equality if $D'=D$. 
In particular, we recall that as $D$ is admissible, it is weakly admissible \cite{Fontaine:representationsemistable}. 

Let us consider the integration pairing as in \cite[Section 7]{Colmez1992}
\[
\varphi^{\rm cris}_A\colon T_p(A) \longrightarrow H^1_{\rm dR}(A)^\vee\otimes_F B_{\rm cris}^+,
\]
which induces on the one hand the crystalline integration on the Tate module of the formal group:
\[
\varphi^{\rm cris}_{\widehat{\cA}}\colon T_p(\widehat{\cA})\longrightarrow D^\vee\otimes_FB_{\rm cris}^+.
\]
and on the other hand Fontaine's integration on the Tate module of the formal group
\[
\varphi_{\widehat{\cA}}\colon T_p(\widehat{\cA})\longrightarrow W^\vee\otimes_F\C_p(1).
\]
where we identify $(\varphi_{\cA})_{|T_p(\widehat{\cA})} $ with $\varphi_{\widehat{\cA}}$. The main result of this subsection is the following.

\begin{theorem}\label{thm:Fontaineformalinjective}
The map $\varphi_{\widehat{\cA}}\colon T_p(\widehat{\cA})\longrightarrow W^\vee\otimes_F\C_p(1)$ is injective.
\end{theorem}

\begin{proof}
The outline of this proof was supplied by Pierre Colmez. 

Let $x\in T_p(\widehat{\cA})$ be such that $\varphi_{\widehat{\cA}}(x)=0$. This means that for every $\omega\in W$, 
\[
\varphi_{\widehat{\cA}}(x)(\omega):=\int_x\omega \pmod {\rm Fil^2(B_{\rm cris}^+)}=0. 
\]
We claim the following:

\begin{lemma}\label{lemma:zeromodfiltration}
Let $x\in T_p(\widehat{\cA})$ and $\omega\in W$. Then, we have that $\varphi^{\rm cris}_{\widehat{\cA}}(x)(\omega)=0$ if and only if $\varphi^{\rm cris}_{\widehat{\cA}}(x)(\omega)=0 \pmod {{\rm Fil}^2(B_{\rm cris}^+)}$.
\end{lemma}

\begin{proof}
We first remark that as $B_{\rm cris}^+\subset B_{\rm dR}^+$ and the filtration of $B_{\rm cris}^+$ is the one induced by the filtration of $B_{\rm dR}^+$, it is enough to prove the lemma with $B_{\rm cris}^+$ replaced by $B_{\rm dR}^+$.

Let us fix $\omega\in W$. Then the maps $\omega^\vee\colon V_p(\widehat{\cA})\to B_{\rm dR}^+$ and $\omega^*\colon V_p(\widehat{\cA})\to B_{\rm dR}^+/I^2$ given by $\omega^\vee(x):=\varphi^{\rm cris}_A(x)(\omega)$ and $\omega^*:=\pi\circ\omega^\vee$, with $\pi\colon B_{\rm dR}^+\longrightarrow B_{\rm dR}^+/I^2$ is the natural projection, are naturally $\Q_p$-linear maps, $G_F$-equivariant. 
Here $I\subset B_{\rm dR}^+$ is the maximal ideal of $B_{\rm dR}^+$, whose powers define the filtration of this ring.
Therefore, we can see $\omega^\vee$ and $\omega^*$ as elements of the following modules:
$$
\begin{array}{ccccccccc}
\omega^\vee&\in&{\rm Hom}_{G_F}\bigl(V_p(\widehat{\cA}), B_{\rm dR}^+ \bigr)&\cong&\bigl(V\otimes_{\Q_p}B_{\rm dR}^+\bigr)^{G_F}\\
\omega^*&\in&{\rm Hom}_{G_F}\bigl(V_p(\widehat{\cA}), B_{\rm dR}^+/I^2 \bigr)&\cong&\bigl(V\otimes_{\Q_p}B_{\rm dR}^+/I^2\bigr)^{G_F}
\end{array}
$$
where we have denoted by $V:=V_p(\widehat{\cA})^\vee$. The projection map $\pi$ induces a natural map 
\[
\pi_0\colon\bigl(V\otimes B_{\rm dR}^+\bigr)^{G_F}\longrightarrow \bigl(V\otimes B_{\rm dR}^+/I^2\bigr)^{G_F}
\] 
and the lemma will follow if we show that this map is an isomorphism.

For every $n\ge 2$ consider the natural exact sequence of rings and ideals
$$
0\longrightarrow I^n/I^{n+1}\longrightarrow B_{\rm dR}^+/I^{n+1}\stackrel{\pi_n}{\longrightarrow}B_{\rm dR}^+/I^n\longrightarrow 0.
$$
We tensor this exact sequence with $V$, over $\Q_p$, and obtain the exact sequence of $G_F$-modules:
$$
0\longrightarrow V\otimes I^n/I^{n+1}\longrightarrow V\otimes B_{\rm dR}^+/I^{n+1}\stackrel{1\otimes \pi_n}{\longrightarrow} V\otimes B_{\rm dR}^+/I^n\longrightarrow 0.
$$
The important remark is that the Hodge-Tate weights of $T$ are $0$ and $-1$, therefore 
$$
V\otimes I^n/I^{n+1}\cong V\otimes \C_p(n)\cong \C_p(n)^a\oplus\C_p(n-1)^b,
$$ 
with positive $n, n-1$ and non-negative integers $a,b$. As a consequence the long exact $G_F$ continuous cohomology sequence gives the exact sequence:
$$
0\longrightarrow \bigl(V\otimes B_{\rm dR}^+/I^{n+1}\bigr)^{G_F}\stackrel{u_n}{\longrightarrow}\bigl(V\otimes B_{\rm dR}^+/I^n \bigr)^{G_F}\longrightarrow H^1(G_F, V\otimes\C_p(n))\longrightarrow \cdots
$$
By Tate's result \cite[Theorem 2]{tate:pdiv}, the morphisms $u_n$ induced by $1\otimes \pi_n$, are isomorphisms for every $n\ge 2$.

For every $k\ge 3$, let $v_k\colon \bigl(V\otimes B_{\rm dR}^+/I^{k} \bigr)^{G_F}\longrightarrow \bigl(V\otimes B_{\rm dR}^+/I^2  \bigr)^{G_F}$ the map $u_2\circ \cdots\circ u_{k-2}\circ u_{k-1}$.
Obviously, $v_k$ is an isomorphism of $K$-vector spaces and we have
\[
\bigl(V\otimes B_{\rm dR}^+  \bigr)^{G_F}\cong \bigl(V\otimes \varprojlim_{n\geq 2}B_{\rm dR}^+/I^n  \bigr)^{G_F}\cong \varprojlim_{n\geq 2}\Bigl(V\otimes B_{\rm dR}^+/I^n  \Bigr)^{G_F}\stackrel{v}{\cong} \bigl(V\otimes B_{\rm dR}^+/I^2\bigr)^{G_F},
\]
where $v$ is induced by the family $(v_n)_{n\ge 2}$ and therefore it is an isomorphism. 
\end{proof}

Now, by the basic properties of $\varphi_A^{\rm cris}$ (cf.~\cite[Proposition 3.1.(iii)]{Colmez1992}), for every $\omega\in W$ and every $n\ge 0$ we have 
\[
\varphi^{\rm cris}_{\widehat{\cA}}(x)(\phi_D^n(\omega))=\int_x\phi^n_D(\omega)=\varphi^n\Bigl(\int_x\omega\Bigr)=0,
\]
where $\varphi:=(\varphi^{\rm cris})^r$ and $\varphi^{\rm cris}$ denotes the Frobenius on $B_{\rm cris}^+.$ 
Let us denote by $H:=\sum_{n\ge 0}\phi_D^{\text{abs},n}(W)\subset D$. 
By construction, we have that $\varphi(H)=H$ and $W\subset H$.

\begin{lemma}\label{lemma:iteratesequalDieudonne}
We have $H=D$.
\end{lemma}

\begin{proof}
We first claim that $\bigl(H, (\phi_{D})_{|H}, {\rm Fil}^\bullet(H)\bigr)$ is a weakly admissible filtered, Frobenius module. It certainly is a filtered, Frobenius submodule of $D$, therefore $t_N(H)\ge t_H(H)$ , by the admissibility of $D$.  
Moreover, the Hodge polygon of $D$ consists of a horizontal segment of length $h-g$, where ${\rm dim}_F(D)=h$ and ${\rm dim}_F(W)=g$ and a segment of slope $1$ and hight $g$. As the filtration of $H$ is given by: ${\rm Fil}^i(H)=H$ for $i\le 0$,
${\rm Fil}^1(H)=W$ and ${\rm Fil}^j(H)=0$ for $j\ge 2$, the Hodge polygon of $H$ has a horizontal segment of length 
$h'-g$, $h'={\rm dim}_F(H)$ and a segment of slope $1$ and height $g$.
So $t_H(H)=t_H(D)=g$.

On the other hand, by \cite[Corollary 5.7.7 \& 5.7.8]{Katz:crystalline}, we have an exact sequence of filtered, Frobenius modules
\[
0\longrightarrow H^1_{\rm et}(\overline{\cA},\mZ_p)\longrightarrow H^1_{\rm cris}(\overline{\cA}/F)\longrightarrow D\longrightarrow 0
\]
where $\overline{\cA}$ denotes the special fiber of the Neron model of $A$.
The above exact sequence identifies the first term with the slope $0$ submodule of $H^1_{\rm cris}(\overline{\cA}/F)$ and $D$ as the slope $>0$ quotient of the same Frobenius module.
As all the slopes of $\phi_D$ on $D$ (so on $H$ as well) are $>0$, we have that 
$t_N(H)\le t_N(D)=t_H(D)=g$. Therefore we get that $t_N(H)=t_H(H)=g$, i.e. $H$ is a indeed a weakly admissible filtered, Frobenius submodule of $D$. Therefore, the quotient $D/H$ (in the category of filtered, Frobenius modules) is weakly admissible.
But $D/H$ has positive slopes of Frobenius and its filtration is: ${\rm Fil}^i(D/H)=D/H$ for $i\le 0$ and ${\rm Fil}^j(D/H)=0$ 
for all $j\ge 1$. So the Hodge polygon of $H$ has only a horizontal segment and as the slopes of Frobenius on $D/H$ are positive, if it is non-zero, it cannot be weakly admissible. So $D/H=0$ i.e. ~$D=H$. 
\end{proof}
As a consequence of \autoref{lemma:zeromodfiltration} and \autoref{lemma:iteratesequalDieudonne}, $\varphi^{\rm cris}_{\widehat{\cA}}(x)(\omega)=0$ for all $\omega\in D$, i.e. , $\varphi^{\rm cris}_{\widehat{\cA}}(x)=0$.
Since $\varphi^{\rm cris}_{\widehat{\cA}}=(\varphi^{\rm cris}_{A})_{|T_p(\widehat{\cA})}$ and $\varphi^{\rm cris}_A$ is injective by \cite[Th\'eor\`eme 5.2.(ii)]{Colmez1992}, we have that $x=0$, and hence our desired result. 
\end{proof}

\subsection{A first application of \autoref{conj:periodic}.}

We keep all the notations from the previous sections and Subsection \ref{subsec:conventions}, i.e.
let $A$ be an abelian variety over $F$ satisfying the assumptions there. Let $K\subset L\subset \Kbar$ be such that $L$ is a finite extension of $K$ and denote $G_L$ the absolute Galois group of $L$.

\begin{lemma}
We have that $\bigl(T_p(A)\bigr)^{G_L}=0$ and ${\rm per}(\cB_A)=(\cB_A)^{G_L}$.
\end{lemma}

\begin{proof}
\autoref{conj:periodic} implies that we have the following commutative diagram with exact rows:
$$
\begin{array}{ccccccccc}
&&0&\longrightarrow&T_p(A)&\stackrel{\varphi_\cA}{\longrightarrow}&{\rm Lie}(A)\otimes\C_p(1)\\
&&\cap&&\cap&&||\\
0&\longrightarrow&{\rm per}(\cB_A)&\longrightarrow&\cB_\cA&\stackrel{\varphi_\cA}{\longrightarrow}&{\rm Lie}(A)\otimes\C_p(1)
\end{array}
$$
Now we take $G_L$-invariants of this diagram and use the fact that
$\left({\rm Lie}(A)\otimes\C_p(1)\right)^{G_L}=0$ and $({\rm per}(\cB_A))^{G_L}={\rm per}(\cB_A)$.
\end{proof}

Of course, it follows, either from the lemma, or using the diagram in the proof of the lemma, that for every finite extension $M$ of $F$ in $\Kbar$, we have $(T_p(A))^{G_M}=0$ and $(\cB_A)^{G_M}=\bigl({\rm per}(\cB_A)\bigr)^{G_M}.$
 
%-----------------------------------------------------------------% 
  \section{\bf A $p$-adic uniformization result for abelian varieties with good reduction} 
 \label{sec:uniformization}
 %-----------------------------------------------------------------%
In this section, we prove our \autoref{thm:main2} which states that the $p$-adic points of an abelian variety $A$ over $F$ with good reduction which satisfies \autoref{assumption}, admit a certain type of $p$-adic uniformization, which has not been observed before and strongly resembles the classical complex uniformization.

We recall from Section \ref{sec:intro} that there is a certain subgroup $A^{(p)}(\Kbar)$ of $A(\Kbar)$ which we are interested in studying. 
We start by recalling the notation $A_{\rm tor}(\Kbar), A_{p{\text{-tors}}}(\Kbar), A_{p'{\text{-tors}}}(\Kbar)$ which denotes the torsion subgroup, the $p$-power torsion subgroup, and the prime-to-$p$ power torsion subgroup of the group $A(\Kbar)$, respectively.
 
We consider the $w$-topology on $A(\Kbar)$, it induces the discrete topology on $A_{p'{\text{-tors}}}(\Kbar)$, therefore this subgroup is closed in $A^w(\Kbar)$. Similarly, the $w$-topology on $\cB_\cA$ induces the discrete topology on ${\rm per}(\cB_\cA)$, and so this subgroup is closed in $\cB_\cA$.

\begin{definition}\label{defn:Apwithw}
We define the topological abelian groups 
\[
A^{(p),w}:=A^w(\Kbar)/A_{p'{\text{-tors}}}(\Kbar)\ \mbox{ and } \ \cB_\cA^{(p)}:=\cB_\cA/{\rm per}(\cB_\cA),
\]
on which we consider the respective quotient topologies.
\end{definition}

\begin{remark}
Moreover, we have natural isomorphisms of groups:
$$A^{(p),w}\otimes_\Z\Q\cong A(\Kbar)\otimes_\Z\Q\cong A(\Kbar)/A_{\rm tor}(\Kbar)\ \quad \mbox{ and } \quad \cB_\cA^{(p)}\cong \cB_\cA\otimes_\Z\Q. 
$$ 
As before the $w$-topology of $A(\Kbar)$ induces the discrete topology on $A_{\rm tor}(\Kbar)$ (\autoref{lemma:winducesdiscretetorsion}), therefore this subgroup is closed in $A^w(\Kbar)$. We define 
on $A^{(p),w}\otimes_Z\Q$ the topology induced by the isomorphisms above, with the quotient topology on $A^w(\Kbar)/A_{\rm tor}(\Kbar)$.
 \end{remark}

We recall from Section \ref{sec:Felliptic} that we have an exact sequence of $G_F$-modules
\begin{equation}\label{eqn:pointSES}
0\to T_p(A)\to \cB_\cA\to A(\Kbar)\to 0. 
\end{equation}
 
We consider the following commutative diagram, with exact rows and columns
$$
\begin{array}{ccccccccccc}
&&&&{\rm per}(\cB_\cA)&=&A_{p'{\text{-tors}}}(\Kbar)\\
&&&&\cap&&\cap\\
0&\longrightarrow&T_p(A)&\longrightarrow&\cB_\cA&\longrightarrow&A(\Kbar)&\longrightarrow&0\\
&&||&&\downarrow&&\downarrow\\
0&\longrightarrow&T_p(A)&\longrightarrow&\cB^{(p)}_\cA&\longrightarrow&A^{(p),w}(\Kbar)&\longrightarrow&0
\end{array}
$$
We mentioned above that we can see $A^{(p)}$ also as a subgroup of $A(\Kbar)$, namely there is a natural section $s\colon A(\mC_p)\to A_{p'{\text{-tors}}}(\Kbar)$ of the natural inclusion $A_{p'{\text{-tors}}}(\Kbar)\subset A(\mC_p)$ defined as follows (see \cite{fontaine_presque}). The exponential is defined locally, i.e. there is an $\cO_F$-lattice $\Lambda$ in $\Lie(\cA)(\cO_F)$ and a continuous map 
\[
\exp_\Lambda\colon \cO_{\C_p}\otimes_{\cO_F}\Lambda\to A(\C_p),
\]
such that $\log_A(\exp_\Lambda(x))=x$ for $x\in \cO_{\C_p}\otimes_{\cO_F}\Lambda$. For $x\in A(\C_p)$, for $n\gg 0$ we have $p^n\log_A(x)\in \cO_{\C_p}\otimes_{\cO_F}\Lambda$ and 
\[
s(x):=p^{-n}\pi\left(\frac{x^{p^n}}{\exp_\Lambda\left(p^n\log_A(x)\right)}\right),
\]
where $\pi\colon A_{\rm tor}(\Kbar)\to A_{p'{\text{-tors}}}(\Kbar)$ is the canonical projection. Above, we denoted (following \cite{fontaine_presque}) $A(\C_p)$ multiplicatively and $A_{p'{\text{-tors}}}(\Kbar)$ additively. Then, $s$ is continuous and $G_K$-equivariant. 
Alternatively,  one can construct the section as follows (we owe this construction to the referee). Write an element $x$ of $A^{\text{Fo}}(\mC_p)$, where $A^{\text{Fo}}(\mC_p)$ is $A(\mC_p)$ is endowed with the natural $p$-adic topology,  in terms of its topologically $p$-power nilpotent part times a prime-to-$p$ torsion element. The section $s$ is given as the limit $\lim_n[p^{n!}]x$ as this limit kills the topologically $p$-power nilpotent part and acts as the identity on the prime-to-$p$ torsion part. We note that the same construction with $A^{\text{Fo}}(\mC_p)$ replaced with $A^w(\Kbar)$ will work even if $A^w(\Kbar)$ is not complete. 

 Let $s\colon A(\C_p)\to A_{p'{\text{-tors}}}(\Kbar)$ denote the section to the inclusion $A_{p'{\text{-tors}}}(\Kbar) \subset A(\C_p)$ defined above. 
We denote by $A^{(p)}(\C_p):=\ker(s)$ and $A^{(p)}(\Kbar):=A(\Kbar)\cap A^{(p)}(\C_p)$.  
We remark that we have natural isomorphisms of groups
 $A^{(p)}(\Kbar) \cong A(\Kbar)/A_{p'{\text{-tors}}}(\Kbar)$,  and therefore we could have defined $A^{(p)}(\Kbar)$ as the subgroup of $A(\Kbar)$ given by $A(\Kbar)\cap \ker(s)$. 
 
We now continue our investigation of the structure of the topological group $A^{(p),w}(\Kbar)$; we will abuse notation and denote by 
\[
\varphi_\cA\colon \cB_\cA^{(p)}\to {\rm Lie}(\cA)(\cO_F)\otimes_{\cO_F}\C_p(1)
\] 
the restriction of $\varphi_\cA$ to $\cB^{(p)}_\cA$. Given that $A$ satisfies \autoref{assumption}, \autoref{conj:periodic1} implies that $\varphi_\cA$ is injective.

We then have a natural commutative diagram of $\Z_p$- and $G_F$-modules 
\[
\begin{tikzcd}[row sep = 1.2em]
0 \arrow{r} & T_p(A) \arrow{r}\arrow[equal]{d} & \cB_{\cA}^{(p)} \arrow{r} \arrow[right hook->]{d}{\varphi_{\cA}} & A^{(p),w}(\Kbar)\arrow{r} \arrow[right hook->]{d}& 0 \\
0 \arrow{r} & \varphi_{\cA}(T_p(A)) \arrow{r} & {\rm Lie}(\cA)(\cO_F)\otimes_{\cO_F}\C_p(1)\arrow{r} & \frac{\left(\Lie(\cA)(\cO_F)\otimes_{\cO_F}\C_p(1)\right)}{\varphi_{\cA}(T_p(A))} \arrow{r} & 0
\end{tikzcd}
\]
where above and throughout this section, equality signs in diagrams will denote isomorphisms. 
We let
\[
\iota_A\colon A^{(p),w}(\Kbar)\hookrightarrow ({\rm Lie}(\cA)(\cO_F)\otimes_{\cO_F}\C_p(1))/\varphi_\cA(T_p(A))
\]
denote the injective, $G_F$-equivariant, $\Z_p$-linear homomorphism induced by the diagram and the fact that 
$(\varphi_\cA)$ is injective. We observe that as $\varphi_\cA$ is continuous with respect to the $w$-topology on $\cB_A^{(p)}$, $\iota_A$ is continuous with respect to the $w$-quotient topology on $A^{(p),w}(\Kbar)$ and the quotient topology on ${\rm Lie}(A)\otimes\C_p(1)/\varphi_A(T_p(A))$. Notice that as 
$T_p(A)$ is compact, $\varphi_A(T_p(A))$ is a closed subgroup of ${\rm Lie}(A)\otimes\C_p(1)$, so the quotient topology makes sense.

Next, we will describe the image $\iota_A$ and then show that the triple 
\[
\left(T_p(A), \Lie(\cA)(\cO_F)\otimes_{\cO_F}\C_p(1), \varphi_\cA\colon T_p(A)\hookrightarrow {\rm Lie}(\cA)(\cO_F)\otimes_{\cO_F}\C_p(1)\right)
\] 
determines the group $A^{(p),w}(\Kbar)$. 
This latter result bears some resemblance to \cite[Theorem B]{scholze_weinstein}. 
Recall the definition of crystalline elements from \autoref{defn:cyrstalline}.

\begin{theorem}
\label{prop:imageiota}
Recall the injective, continuous, $G_F$-equivariant, $\Z_p$-linear homomorphism
\[
\iota_A\colon A^{(p),w}(\Kbar)\hookrightarrow ({\rm Lie}(\cA)(\cO_F)\otimes_{\cO_F}\C_p(1))/\varphi_\cA(T_p(A)).
\]
The image of $\iota_A$ corresponds to the crystalline elements of $ \left(\Lie(\cA)(\cO_F)\otimes_{\cO_F}\C_p(1)\right)/\varphi_{\cA}(T_p(A))$. 
\end{theorem}

\begin{proof}
As $\C_p$ is a field, $\C_p(1)$ and $\left(\Lie(\cA)(\cO_F)\otimes_{\cO_F}\C_p(1)\right)$ are divisible groups.
Let us first determine the torsion of $\left(\Lie(\cA)(\cO_F)\otimes_{\cO_F}\C_p(1)\right)/\varphi_\cA(T_p(A))$.
We have the following isomorphisms of $\Z_p$ and $G_F$-modules:
\begin{align*}
\left(\left(\Lie(\cA)(\cO_F)\otimes_{\cO_F}\C_p(1)\right)/\varphi_\cA(T_p(A))\right)[p^n] &\cong  \left(\frac{1}{p^n}\varphi_\cA(T_p(A))\right)/\varphi_\cA(T_p(A)) \\
&\cong \varphi_\cA\left(\left(\frac{1}{p^n}T_p(A)\right)/T_p(A)\right) \\
& \cong\iota_A\left(A(\Kbar)[p^n] \right).
\end{align*}
This implies that
\[
\left(\left(\Lie(\cA)(\cO_F)\otimes_{\cO_F}\C_p(1)\right)/\varphi_{\cA}(T_p(A))\right)[p^{\infty}]\cong (\varphi_\cA \otimes 1_{\Q_p})(V_p(A))/\varphi_\cA(T_p(A))\cong \iota_A\left(A(\Kbar)[p^\infty]\right),
\] 
where we recall that $V_p(A)=T_p(A)\otimes_{\Z}\Q\subset \cB_\cA^{(p)} \otimes_{\Z} \Q$. 
Moreover, $\iota_A$ induces a $G_F$-isomorphism between the $p$-power torsion of the two modules. 
From this, we deduce that every torsion point of $\left(\Lie(\cA)(\cO_F)\otimes_{\cO_F}\C_p(1)\right)/\varphi_{\cA}(T_p(A))$ is algebraic (and, by definition crystalline).

We now compare the non-torsion points. 
Let $x\in \left(\Lie(\cA)(\cO_F)\otimes_{\cO_F}\C_p(1)\right)/\varphi_\cA(T_p(A))$ be a non-torsion point. 
We consider its image, which we denote also by $x$, in 
\[
\faktor{\left(\left(\Lie(\cA)(\cO_F)\otimes_{\cO_F}\C_p(1)\right)/\varphi_\cA(T_p(A))\right)}{\left(\left(\Lie(\cA)(\cO_F)\otimes_{\cO_F}\C_p(1)\right)/\varphi_\cA(T_p(A))\right)[p^\infty]} 
\]
which is isomorphic to $\left(\Lie(\cA)(\cO_F)\otimes_{\cO_F}\C_p(1)\right)/(\varphi_\cA \otimes 1_{\Q_p})(V_p(A))$. 
Recall that we have the diagram:
\[
\begin{tikzcd}[row sep = 1.2em]
0 \arrow{r} & V_p(A) \arrow{r} \arrow[equal]{d}{(\varphi_\cA \otimes 1_{\Q_p})} & \cB^{(p)}_\cA\otimes_{\Z} \Q \arrow{r} \arrow[right hook->]{d}{(\varphi_\cA \otimes 1_{\Q_p})} &A^{(p),w}(\Kbar)\otimes_{\Z} \Q \arrow{r}\arrow[right hook->]{d}{\iota_A} &0\\
0 \arrow{r} & (\varphi_\cA \otimes 1_{\Q_p})(V_p(A)) \arrow{r}\arrow[equal]{d}  & \Lie(\cA)(\cO_F) \otimes_{\cO_F} \C_p(1) \arrow{r} &\frac{\left(\Lie(\cA)(\cO_F)\otimes_{\cO_F}\C_p(1)\right)}{(\varphi_\cA \otimes 1_{\Q_p})(V_p(A))} \arrow{r} &0\\
0 \arrow{r} & (\varphi_\cA \otimes 1_{\Q_p})(V_p(A)) \arrow{r} & \alpha^{-1}(x\Q_p) \arrow{r}\arrow[right hook->]{u} & x\Q_p \arrow{r} \arrow[right hook->]{u}&0. 
\end{tikzcd}
\]

Let $L$ be a finite extension of $F$ contained in $\Kbar$ such that $x$ is $G_L$-invariant, and consider the inclusion $$A^{(p),w}(L)\otimes_\Z\Q:=\bigl(A^{(p),w}(\Kbar) \bigr)^{G_L}\otimes_\Z\Q= A(\Kbar)^{G_L}\otimes_\Z\Q=A(L)\otimes_\Z\Q\subset A(\Kbar)\otimes_\Z\Q=A^{(p),w}(\Kbar)\otimes_\Z\Q.$$
It induces, by pull-back of the first row, another commutative diagram with exact rows:
\[
\begin{tikzcd}[row sep = 1.2em]
0 \arrow{r} & V_p(A) \arrow{r} \arrow[equal]{d}{(\varphi_\cA \otimes 1_{\Q_p})} & \cB^{(p)}_\cA(L)\otimes_{\Z} \Q \arrow{r} \arrow[right hook->]{d}{(\varphi_\cA \otimes 1_{\Q})} &A^{(p),w}(L)\otimes_{\Z} \Q \arrow{r}\arrow[right hook->]{d}{\iota_A} &0\\
0 \arrow{r} & (\varphi_\cA \otimes 1_{\Q_p})(V_p(A)) \arrow{r}\arrow[equal]{d}  & \Lie(\cA)(\cO_K) \otimes_{\cO_K} \C_p(1) \arrow{r} &\frac{\left(\Lie(\cA)(\cO_K)\otimes_{\cO_K}\C_p(1)\right)}{(\varphi_\cA \otimes 1_{\Q_p})(V_p(A))} \arrow{r} &0\\
0 \arrow{r} & (\varphi_\cA \otimes 1_{\Q_p})(V_p(A)) \arrow{r} & \alpha^{-1}(x\Q_p) \arrow{r}\arrow[right hook->]{u} & x\Q_p \arrow{r} \arrow[right hook->]{u}&0. 
\end{tikzcd}
\]

We wish to
consider the long exact, {\bf continuous} $G_L$-cohomology diagram.
The first row is an exact row of finite dimensional $\Q_p$-vector spaces and the $w$-topologies induced by $\cB_A^{(p)}\otimes_{\Z}\Q$ and respectively $A^{(p)}(\Kbar)\otimes_\Z\Q$ on the middle and third terms respectively are the natural $p$-adic topologies. Therefore this row has a natural, continuous splitting as $\Q_p$-vector spaces.

The second row is an exact sequence of $\Q_p$-Banach spaces, and so these Banach spaces are orthonormalizable. By choosing an orthonormal basis of $\Lie(\cA)(\cO_F)\otimes_{\cO_F}\C_p(1)$, indexed by a totally ordered set, such that the first $2g$ basis elements are the images of a $\Q_p$-basis of $V_p(A)$, we obtain a continuous splitting as $\Q_p$-Banach spaces of the second row.

Therefore we consider the long exact, continuous $G_L$-cohomology of the diagram
\[
\begin{tikzcd}[row sep = 1.2em]
0 \arrow{r} & A(L)\otimes_{\Z}\Q \arrow{r}{\partial} \arrow[right hook->]{d}{\iota_{A}} & H^1(L, V_p(A) \arrow{r} \arrow[equal]{d}& {} \\
0 \arrow{r} & \left(\frac{\left(\Lie(\cA)(\cO_F)\otimes_{\cO_F}\C_p(1)\right)}{(\varphi_\cA \otimes 1_{\Q_p})(V_p(A))}\right)^{G_L} \arrow[equal]{r} & H^1\left(L, (\varphi_{\cA}\otimes 1_{\Q_p})(V_p(A)\right)  \arrow{r} & 0 \\
0 \arrow{r} & x\Q_p \arrow{r}{\gamma} \arrow[right hook->]{u} & H^1\left(L, (\varphi_{\cA}\otimes 1_{\Q_p})(V_p(A)\right) \arrow{r}\arrow[equal]{u} &{}
\end{tikzcd} 
\]
where the middle isomorphism follows from \cite[Theorem 2]{tate:pdiv}. 
As we already mentioned, \cite[Example 3.11 \& Proposition 5.4]{bloch_kato} implies that $\partial\left(A(L)\otimes_{\Z}\Q\right)=H^1_f(L, V_p(A))$, and therefore $x\in \iota_A\left(A(L)\otimes_{\Z}\Q\right)$ if an only if $\gamma(x)\in H^1_f(L, V_p(A))$. By \autoref{rem:crystalline}, this is equivalent to $\alpha^{-1}(x\Q_p)$ being a crystalline $G_L$-representation, and hence $x$ being crystalline. 
\end{proof}

\begin{remark}
It follows from the proof of \autoref{prop:imageiota} that, for every finite extension $L$ of $F$, we have natural isomorphisms: 
\[
\beta_L\colon \left(\left(\Lie(\cA)(\cO_F)\otimes_{\cO_F}\C_p(1)\right)/\varphi_{\cA}(T_p(A))\right)^{G_L}\cong H^1\left(L, \varphi_{\cA}(T_p(A))\right),
\]
which are compatible with restrictions from $L'$ to $L$, if $L\subset L'$.
\end{remark}

\appendix
\section{\bf{Another proof of the injectivity of the Fontaine integral}}
\label{appendix}
Let, as before, $F$ be a finite, unramified extension of $\Q_p$, 
$K$ the maximal unramified extension of $F$ in $\Kbar$ and $G_F, G_K$ the absolute Galois groups of $F$ and respectively $K$. Let 
$A$ be an abelian variety with good reduction over $F$, $\cA$ its N\'eron model over ${\rm Spec}(\cO_F)$ and $T_p(A)$ its $p$-adic Tate module. In this Appendix we present two proofs of the following theorem. 

\begin{theorem}\label{thm:Fontaineunramified}
If $A$ satisfies the property $T_p(A)^{G_K}=0$, then the Fontaine integral $\varphi_A\colon T_p(A)\longrightarrow {\rm Lie}(A)(K)\otimes_K\C_p(1)$ is injective.
\end{theorem}

As $K$ is a discretely valued field, which is not complete, we have to first complete it, say we denote $M$ the completion of $K$ seen as a subfield of $\C_p$ (which denotes, as before, the completion of $\Kbar$). Let $\barM$ denote the algebraic closure of $M$ in $\C_p$, then $\barM$ is an algebraically closed field containing $\Kbar$. Moreover, a simple application of Krasner's lemma (see also \cite{IovitaZaharescu:CompletionsRAT}) implies that $M\cap \Kbar=K$.
Therefore, if we denote $G_M:={\rm Gal}(\barM/M)$, we have a natural, continuous group homomorphism induced by restriction:~$\gamma\colon G_M\to G_K$.

\begin{lemma}\label{lemma:gammaisom}
$\gamma$ is an isomorphism.
\end{lemma}

\begin{proof}
We have natural isomorphisms induced by restriction, namely 
$\alpha\colon{\rm Aut}_{\rm cont}(\C_p/K)\cong G_K$ and $\beta\colon{\rm Aut}_{\rm cont}(\C_p/M)\cong G_M$. Moreover $\gamma=\alpha\circ \beta^{-1}$.
\end{proof}

Let us now denote by $A_M$ the base-change of the abelian variety $A$ to $M$. Let us observe that if $T_p(A_M)$ is the $p$-adic Tate module of $A_M$, as $A$ is defined over $F$ and
$\Kbar\subset \barM$, we have the equality $T_p(A_M)=T_p(A)$ as $\Z_p$-modules and $G_M$-modules, where the action of $G_M$
on $T_p(A)$ is via $\gamma$. From now on we will identify these two Tate modules.

On the other hand, we have the $\cO_{\Kbar}$-module $\Omega:=\Omega^1_{\cO_{\Kbar}/\cO_K}$ and the $\cO_{\barM}$-module 
$\Omega_M:=\Omega^1_{\cO_{\barM}/\cO_M}$ and a natural $\cO_{\Kbar}$-linear map: $f\colon\Omega\to\Omega_M$ induced by the 
inclusions $\cO_{\Kbar}\subset \cO_{\barM}$ and $\cO_K\subset \cO_M$.

\begin{lemma}\label{lemma:differentialisom}
The morphism $1_{\cO_{\barM}}\otimes f\colon\cO_{\barM}\otimes_{\cO_{\Kbar}}\Omega\to \Omega_M$
is a $\cO_{\barM}$, $G_M$-equivariant (the action of $G_M$ is semi-linear) isomorphism.
\end{lemma}

\begin{proof}
Both modules are generated over $\cO_{\barM}$ by the family 
\[ 
\left(\frac{d\zeta_n}{\zeta_n}\right)_{n\ge 1},
\]
with unique relations 
\[
p\left(\frac{d\zeta_{n+1}}{\zeta_{n+1}}\right)=\frac{d\zeta_n}{\zeta_n}
\] 
for all $n\ge 1$ where $(\zeta_n)_{n\ge 1}$ is a compatible family of $p^n$-th roots of $1$ in $\Kbar$. This implies that $1_{\cO_{\barM}}\otimes f$ is an isomorphism.
\end{proof}

We denote by $\varphi_{A_M}\colon T_p(A_M)\to {\rm Lie}(A_M)(M)\otimes_M\C_p(1)$ Fontaine's integration map for $A_M$, where here $\C_p=V_p(\Omega_M)=V_p(\Omega)$. Note that we have the following diagram:
\[
\begin{array}{cccccccccc}
T_p(A_M)&\stackrel{\varphi_{A_M}}{\longrightarrow}&{\rm Lie}(A_M)(M)\otimes_M\C_p(1)\\
||&&\uparrow \cong\\
T_p(A)&\stackrel{\varphi_A}{\longrightarrow}&{\rm Lie}(A)(K)\otimes_K\C_p(1)
\end{array}
\]
It is obvious, given the definitions and \autoref{lemma:gammaisom} and \autoref{lemma:differentialisom}, that this diagram is commutative, and hence to prove the theorem it is enough to prove that $\varphi_{A_M}$ is injective.

Therefore we can re-denote things as follows:~let $K$ be {\bf the completion} of $\mQ_p^{\rm{ur}}$, $\Kbar$ an algebraic closure of $K$, $\C_p$ the completion of $\Kbar$ and $A$ an abelian variety with good reduction over $K$ satisfying 
$T_p(A)^{G_K}=0$. 

\begin{theorem}[Yeuk Hay Joshua Lam, Alexander Petrov]
\label{thm:fontcompleted}
In the notations above, if $\varphi_A\colon T_p(A)\to {\rm Lie}(A)(K)\otimes_K\C_p(1)$ denotes Fontaine's integration map, then $\varphi_A$ is injective.
\end{theorem}

The proof of \autoref{thm:fontcompleted} presented below was sent to us, independently, by Yeuk Hay Joshua Lam and Alexander Petrov. By the above remarks \autoref{thm:fontcompleted} implies \autoref{thm:Fontaineunramified}.

\begin{proof}
Let $S:=\ker(\varphi_{A}\otimes_{\Z_p}\Q_p)$.
It suffices to show that $S=0$. 
As $V_p(A)$  is a crystalline representation of $G_K$, $S$ is a crystalline representation with respect to the same Galois group.
Consider the following commutative diagram with the first and last rows exact:
$$
\begin{array}{cccccccccc}
0&\longrightarrow&S&\stackrel{f}{\longrightarrow}&V_p(A)&\stackrel{\varphi_A}{\longrightarrow}&{\rm Lie}(A)\otimes_K\C_p(1)\\
&&\cap&&\cap&&||\\
0&\longrightarrow&S\otimes_{\Q_p}\C_p&\stackrel{f\otimes 1}{\longrightarrow}&V_p(A)\otimes_{\Q_p}\C_p&\stackrel{\varphi_A\otimes 1}{\longrightarrow}&{\rm Lie}(A)\otimes_K\C_p(1)& & \\
&&\cap&&||&&||\\
0&\longrightarrow&(\C_p)^g&\longrightarrow&V_p(A)\otimes_{\Q_p}\C_p&\stackrel{\varphi_A\otimes 1}{\longrightarrow}&{\rm Lie}(A)\otimes_K\C_p(1)&\longrightarrow&0
\end{array}
$$
Observe that the middle row is not exact but we have $(\varphi_A\otimes 1)\circ (f\otimes 1)=0$, therefore $(f\otimes 1)(S\otimes_{\Q_p}\C_p)\subset {\rm ker}(\varphi_A\otimes 1)=(\C_p)^g$, i.e. we have the maps relating the second to the third rows of the diagram.

It follows from Tate's result \cite[Theorem 2]{tate:pdiv} that $S$ is a $p$-adic representation of $G_K$, which is Hodge-Tate with all Hodge-Tate weights $0$. 
By \cite[Corollary]{Sen:ContinuousCohomology}, this implies that the image of the Galois representation $\rho_S\colon G_K\longrightarrow {\rm End}_{\Q_p}(S)$ is finite. 
Let then $L/K$ be the finite, Galois extension, obviously totally ramified, such that $\rho_S$ factors through $\rho_S\colon {\rm Gal}(L/K)\longrightarrow {\rm End}_{\Q_p}(S)$.  

We wish to show that $S = 0$.  As $S$ is a crystalline $G_K$-representation we have
\[
\begin{array}{ll}
D_{\rm cris}(S) &=\bigl(S\otimes_{\Q_p}B_{\rm cris} \bigr)^{G_K}=
\Bigl(\bigl(S\otimes_{\Q_p}B_{\rm cris}  \bigr)^{G_L}  \Bigr)^{{\rm Gal}(L/K)} = \Bigl(S\otimes_{\Q_p}(B_{\rm cris})^{G_L}\Bigr)^{{\rm Gal}(L/K)} \\
&= S^{{\rm Gal}(L/K)}\otimes_{\Q_p} K.
\end{array}
\]
The last equality follows from \cite[Proposition 5.1.2]{Fontaine:PadicRep}. 
Therefore, from the above we have ${\rm dim}_{\Q_p}(S)={\rm dim}_K(D_{\rm cris}(S))={\rm dim}_{\Q_p}(S^{{\rm Gal}(L/K)}).$ 
In other words, we have that $S=S^{{\rm Gal}(L/K)}=S^{G_{K}}.$
But $S\subset V_p(A)$ and $V_p(A)^{G_K}=0$, therefore, $S=0$.
 \end{proof}

\begin{remark}
The proof of \autoref{thm:fontcompleted} shows that, even without the assumption $V_p(A)^{G_K}=0$, we have $\ker(\varphi_{A}\otimes_{\Z_p}\Q_p) = V_p(A)^{G_K}$. Therefore, in this general case we could have defined $A(\Kbar)^{(p)}:=A(\Kbar)/A_{\rm tor}(K)$.
Using the ideas of this article one might be able to obtain a $G_K$-equivariant map $A(\Kbar)^{(p)}\lra ({\rm Lie}(A)(K)\otimes_K\C_p(1))/\varphi_A(T_p(A))$ and study it. We did not pursue this idea in this paper and only studied the situation when $V_p(A)^{G_K}=0$.
\end{remark}

  \bibliography{refs}{}
\bibliographystyle{amsalpha}

 \end{document}